\documentclass[a4paper,draft,12pt,twoside]{article}

\usepackage{amssymb,amsmath,amscd,amsfonts}
\usepackage{latexsym}
\usepackage{amscd}
\usepackage[all]{xy}
\usepackage{amsthm}
\usepackage{a4}
\usepackage{verbatim}

\newtheorem{prop}{Proposition}[section]
\newtheorem{theorem}[prop]{Theorem}
\newtheorem{lemma}[prop]{Lemma}
\newtheorem{cor}[prop]{Corollary}

\newtheorem{defn}[prop]{Definition}
\theoremstyle{definition}

\newtheorem{exa}[prop]{Examples}
\theoremstyle{remark}

\newcommand{\Hh}{\mathcal{H}}

\newcommand{\Ee}{\mathcal{E}}
\newcommand{\Ll}{\mathcal{L}}

\newcommand{\subs}{\subseteq}

\numberwithin{equation}{section}
\title{Multiply generated dynamical systems and the duality of higher rank graph algebras}
\author{Ionel Popescu \\
{\small  School of Mathematics, Georgia Institute of Technology, Atlanta, GA 30332-0160}\\
{\small Institute of Mathematics "Simion Stoilow", 21, Calea Grivitei Street}\\
{\small 010702-Bucharest, Sector 1, ROMANIA}\\
{\small ipopescu@ math.gatech.edu}\\
Iulian Popescu\\
{\small Institute of Mathematics "Simion Stoilow", 21, Calea Grivitei Street}\\
{\small 010702-Bucharest, Sector 1, ROMANIA}\\
{\small iulian.popescu@imar.ro}}
\date{}

\begin{document}
\maketitle
\footnote{Supportted by the grant 2-CEx06-11-34/2006}

\begin{abstract}
We define a semidirect product groupoid of a system of partially defined local homeomorphisms $T=(T_{1},\ldots, T_{r})$. We prove that this construction gives rise to amenable groupoids. The associated algebra  is a Cuntz-like algebra. We use this construction for higher rank graph algebras in order to give a topological interpretation for the duality in $E$-theory between $C^{*}(\Lambda)$ and $C^{*}(\Lambda^{op})$.
\end{abstract}

\section*{Introduction}
Toeplitz algebras have been used to define extensions of $C^{*}$-algebras. The beginning of this paper is in \cite{pz} where a Toeplitz algebra was the main tool in constructing a $K$-homology class for higher rank graph algebras. Let $(\Lambda,\sigma)$ be a higher rank graph with shape $\sigma$ (see \cite{kupa}), $\Lambda^{*}$ the set of morphism of nonzero shape and $\overline{\Lambda}=\{\Omega\}\cup\Lambda^{*}$ where $\Omega$ is a symbol (the vacuum morphism) which does not belong to $\Lambda^{*}$. We define left and right creations on the Fock space $\mathsf{F}_{\Lambda} = \mathsf{F}=l^{2}(\overline{\Lambda})$:
$$
L_{\lambda}\delta_{\mu}=\left\{ \begin{array}{ll}
\delta_{\lambda\mu}& \textrm{if } s(\lambda)=t(\mu)\\
0 & \textrm{otherwise,}
\end{array} \right.
$$
$$
R_{\lambda}\delta_{\mu}=\left\{ \begin{array}{ll}
\delta_{\mu\lambda}& \textrm{if } s(\mu)=t(\lambda)\\
0 & \textrm {otherwise.}
\end{array} \right.
$$
$$
L_{\lambda}\Omega=R_{\lambda}\Omega=\delta_{\lambda}
$$
The left sided Toeplitz algebra is $\mathcal{E}_{l}=C^{*}(L_{\lambda}; \lambda \in \Lambda^{*})$ the right sided Toeplitz algebra is $\mathcal{E}_{r}=C^{*}(R_{\lambda}; \lambda \in \Lambda^{*})$ and the two sided Toeplitz algebra is $\mathcal{E}=C^{*}(L_{\lambda}, R_{\lambda}; \lambda \in \Lambda^{*})$. These algebras are well known in the rank one case since they give rise to short exact sequences. For example, one can use $\mathcal{E}$ to define
$$0\rightarrow \mathcal{K}\rightarrow \mathcal{E}\rightarrow O_{E}\otimes O_{E^{op}}\rightarrow 0$$
where $E$ is an ordinary oriented graph with certain conditions which give the uniqueness of the algebras $O_{E}$ and $O_{E^{op}}$, $E^{op}$ is the graph obtain by reversing the arrows. In the higher rank case the algebra $\mathcal{E}$ has a more complicated ideal structure. It is this fact which gives a longer exact sequence. It would give a $KK$-class if the sequence was semisplit \cite{pz,ze}. This would follow immediately from the nuclearity of $\mathcal{E}$. We tried to give a more conceptual construction of this algebra as a groupoid algebra of an amenable groupoid but we have not been able to do it. However, the construction of the semidirect product appearing in \cite{re2} can be generalized to give a groupoid description of one-sided Toeplitz algebras of a higher rank graph. We believe that this construction can be generalized to actions of other semigroups. For example actions of cones in $\mathbb{R}^{n}$ should give algebras generated by Wiener-Hopf operators (see \cite{mr, nica}).

The organization of this paper is as follows. In the first section we give the main construction of this paper, the semidirect product of an action of $\mathbb{N}^{r}$ by partial local homeomorphism which we also call a multiply generated dynamical system (MGDS). Such an action is obtained by choosing $r$ commuting partially defined local homeomorphisms. The main examples come from shifts on finite, semifinite and infinite paths of a higher rank graph. There are two groupoids associated to such an action, the groupoid of germs of the pseudogroup generated by these local homeomorphism and the semidirect product groupoid. We think that these groupoids have been around in the study of Toeplitz algebras of higher rank graphs as well as in the general theory of crossed products by partial actions (see for example \cite{er, fmy}). The semidirect product can be defined only if we have a condition on the domains of the partial maps. Otherwise it may not even give rise to an algebraic groupoid. This condition is fullfiled for one sided Toeplitz algebras associated to higher rank graphs. Following the lines in \cite{re2}, we prove that these two groupoids are isomorphic if and only if the dynamical system is essentially free. In the general situation we can prove that the natural map from the semidirect product to the groupoid of germs induces a map of corresponding reduced $C^{*}$-algebras. Next we prove the amenability of the semidirect product. This is done by decomposing the action in subaction and then applying a result on the amenability of extensions.

In the second section we remind several facts about duality in a bivariant theory. We give the definition of Spanier-Whitehead duality and a condition which is often taken as definition in literature. We remind then a way to construct exact sequences starting from an algebra and a tuple of ideals. We show briefly how we can improve the results in \cite{pz} for any locally finite higher rank graph with finite set of objects, regardless the uniqueness of the generating relation of our graph algebras as for examples the graph $\mathbb{N}^{r}$. We obtain a duality for the universal $C^{*}$-algebras $C^{*}(\Lambda)$ and $C^{*}(\Lambda^{op})$..

In the last section we give a groupoid approach to the duality of higher rank graph algebras.
The $K$-theory fundamental class is given by $r$ partial unitaries which can be best described in the groupoid picture as two-sided shifts. Even if we do not have a conceptual groupoid approach to the two-sided Toeplitz algebra, we can define it as a groupoid of germs. The main problem in understanding better this groupoid is the unit space $X$ which is given by the spectrum of the diagonal algebra $\mathcal{E}\cap l^{\infty}(\overline{\Lambda})$. The isometries $L_{\lambda}^{*}R_{\mu}$ with $\sigma(\lambda)=\sigma(\mu)=e_{j}$ are not given by something like shift equivalence. However, the partial isometries can be viewed as partial homeomorphisms using Gelfand duality. The space $X$ has three properties inherited from the graph $\Lambda$:
\begin{enumerate}
 \item  
the shape $\sigma$ extends to a map from $X$ to $\overline{\mathbb{N}}^{r}$
\item
the multiplication $\lambda\mu\nu$ extends to a multiplication $\lambda x\mu$
\item
the unique factorization $\alpha=\lambda\alpha'\mu$ where $\sigma(\alpha)\geq \sigma(\lambda)+\sigma(\mu)$ extends to a unique factorization $x=\lambda x'\mu$ where $\sigma(x)\geq \sigma(\lambda)+\sigma(\mu)$.
\end{enumerate}
These properties are enough to define a MGDS given by shifts on a subspace  of $X\times \Lambda^{\infty}$ together with an equivariant map to the MGDS $\overline{\mathbb{N}}^{r}\times \Lambda^{\infty}$ given again by shifts. This map induces a map of the semidirect product groupoids which in turn induces a morphism between the algebras $\mathcal{T}^{\otimes r}\otimes C^{*}(\Lambda)$ and $\mathcal{E}\otimes C^{*}(\Lambda)$. This morphism is the crucial step in the proof of the duality.

Finally, we want to draw the attention to the results in \cite{em1,em2}. Our groupoid approach may give a hint to what a higher rank hyperbolic group would be.

\section{Cuntz-like algebras associated with multiply generated dynamical systems}   \label{cuntzlike}
Two partial maps on a set $X$, $L:dom(L)\to ran(L)$ and $R:dom(R)\to ran(R)$ can be composed if one defines
$$\textup{dom}(LR)=\{x\in \textup{dom}(R); R(x)\in \textup{dom}(L)\}$$
and $LR(x)=L(R(x)$ for any $x\in dom(LR)$. We say that two partial maps $L$ and $R$ on a set $X$
commute if $dom(LR)=dom(RL)$ and $LR=RL$ on $dom(LR)$.
We set $L^{0}=id$ for any partial map on $X$. If $L$ is injective,
we denote by $L^{-1}$ the partial map
$L^{-1}:ran(L)\to dom(L)$. A local homeomorphism is a map $\phi:X\rightarrow Y$ with the property that each point $x$ has a neighborhood $U$ such that $\phi \vert_{U}:U\rightarrow \phi(U)$ is a homeomorphism. 

A partial homeomorphism on $X$ is a local homeomorphism $S$ from an open set $dom(S)$ of $X$ onto an open set $ran(S)$ of $X$. A set $\mathcal{G}$ of partial homeomorphisms of $X$ which is closed under composition, inversion and containing the identity is called a pseudogroup (\cite{azr}). For any set $S$ of partial homeomorphisms on $X$, there exists the smallest pseudogroup $[S]$ generated by $S$. The semi-direct product groupoid $X\rtimes \mathcal{G}$ of a pseudogroup $\mathcal{G}$ is the set of triples $(x,S,y)$ where $S\in \mathcal{G}, y\in dom(S), x=S(y)$ with the obvious operations 
$$(x,S,y)(y,T,z)=(x,ST,y), (x,S,y)^{-1}=(y,S^{-1},x).$$
The topology is given by the product topology of $X$ and $\mathcal{G}$ where $\mathcal{G}$ has the discrete topology. The groupoid of germs is a quotion of the semidirect product groupoid by the equivalence relation $(x_{1},S_{1},y_{1}) \sim (x_{2},S_{2},y_{2})$ if and only if $y_{1}=y_{2}$ and $S_{1}=S_{2}$ on a neighborhood of $y_{1}$. The topology is the quotient topology. These two groupoids are r-discrete, that is the range and source maps are local homeomorphisms.

In the monoid $\mathbb{N}^{r}$ we write $n=(n_{1},\ldots, n_{r})$ with $n_{j}\in \mathbb{N}$ and $e_{k}$ the coordinates $(0,\ldots,1,\ldots,0)$.
\begin{defn}(Conform \cite{re2})
A multiply generated dynamical system (MGDS) is a pair $(X,T)$ where $X$ is a topological space and  $T=(T_{1},\ldots T_{r})$
a system of $r$ commuting partial homeomorphisms on $X$.
\begin{exa} \label{example}
\begin{enumerate}
\item (\cite{fmy}, Definition 5.1) For $m\in\overline{\mathbb{N}}^{r}$ let $\mathbb{N}^{r}_{m}$ be the higher rank graph $\{(n,n')\in \mathbb{N}^{r}\times\mathbb{N}^{r}:n\leq n'\leq m\}$ where $s(n,k)=k$, $t(n,k)=n$, $\sigma(n,k)=k-n$ and the composition is given by $(n,k)(k,p)=(n,p)$. 
Let $\Lambda$ be a finitely aligned rank $r$ graph and 
$$X_{\Lambda}=\{x:\mathbb{N}^{r}_{m}\to \Lambda; m\in \overline{\mathbb{N}}^{r}_{m}\}$$
 the space of finite, semifinite and infinite paths. The shape $\sigma$ can be extended to $X_{\Lambda}$, $\sigma(x)=m$ where $x$ is defined on $\mathbb{N}^{r}_{m}$.
For $r=2$ an element in $X_{\Lambda}$ can be seen graphically as one of the following:
$$
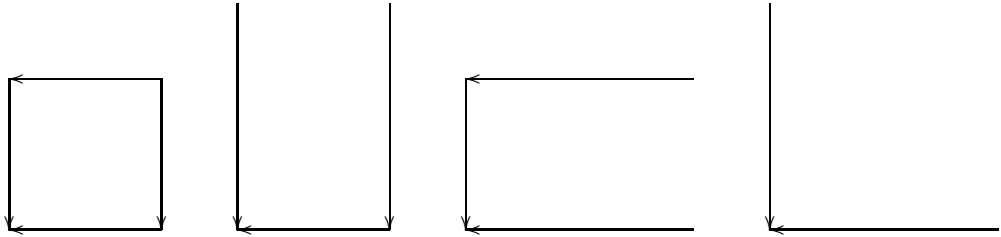
\xy
 {\ar (20,0)*{}; (0,0)*{}};
 {\ar (20,20)*{}; (20,0)*{}};
 {\ar (20,20)*{}; (0,20)*{}};
 {\ar (0,20)*{}; (0,0)*{}};
 {\ar (50,0)*{}; (30,0)*{}};
 {\ar (50,30)*{}; (50,0)*{}};
 {\ar (30,30)*{}; (30,0)*{}};
 {\ar (90,0)*{}; (60,0)*{}};
 {\ar (60,20)*{}; (60,0)*{}};
 {\ar (90,20)*{}; (60,20)*{}};
 {\ar (130,0)*{}; (100,0)*{}};
 {\ar (100,30)*{}; (100,0)*{}};
\endxy
$$
A basis of a topology on $X_{\Lambda}$ is given by the sets $\{x:x(0,k)=\lambda,k\leq\sigma(x)\leq m\}$ where $k\in\mathbb{N}^{r}_{m}$, $\lambda\in \Lambda$ and $m\in \overline{\mathbb{N}}^{r}_{m}$.
 Then the partial homeomorphisms $T_{k}$ with $dom(T_{k})=\{x\in X_{\Lambda}:\sigma(x)\geq k\}$, $T_{k}(x):\mathbb{N}^{r}_{\sigma(x)-k}\to \Lambda$, $T_{k}(n,n')=x(n+k,n'+k)$ give a MGDS.
\item
In the example above the restriction to boundary paths $\partial X$ gives a subsystem. If $\Lambda$ has no sources then $\partial X=\Lambda^{\infty}=\{x: \sigma(x)=(\infty\ldots\infty)\}$ (see \cite{fmy} Definition 5.10 and \cite{kupa})
\item
Let $X$ be the free monoid on $r$ letters $a_1,\ldots,a_r$, that is the set of words $a_{i_1}\ldots a_{i_k}$. Put on $X$ the discrete topology and define $T=(L,R)$ the translations on the left and on the right: $dom(L)=dom(R)=\{a_{i_1}\ldots a_{i_k}$: $k\geq 1\}$ (the set of nonvoid words), $L(a_{i_1}a_{i_2}\ldots a_{i_k})=a_{i_2}\ldots a_{i_k}$, $R(a_{i_1}\ldots a_{i_{k-1}}a_{i_k})=a_{i_1}\ldots a_{i_{k-1}}$. For instance, if $r=1$ then $X=\mathbb{N}$ and $dom(L)=dom(R)=\mathbb{N}\setminus\{0\}$, $L(k)=R(k)=k-1$. It is clear that $L$ and $R$ commute.
\end{enumerate}
\end{exa}
\end{defn}
We denote by $\mathcal{G}(X,T)$ the full pseudogroup generated by the restrictions of $T_{j}\vert_{U}$
where $U$ is an open subset of $X$ on which $T_{j}$ is injective.
Because of the commutation conditions on $T$,
we can define $T^{n}=T_{1}^{n_{1}}T_{2}^{n_{2}}\cdots T_{k}^{n_{r}}$ for $n\in \mathbb{N}^{k}$. One can see a MGDS as an action of the semigroup $\mathbb{N}^{r}$ on $X$ by partial local homeomorphisms. We write sometimes $T_{n}$ instead of $T^{n}$ when we want to view $T$ as a semigroup. We need a technical condition on the domains of $T^{n}$ (domain condition):
$$\textrm{(DC) \;\;\;dom}(T^{n})\cap \textrm{dom}(T^{m})\subset \textrm{dom}(T^{n\vee m})$$
where ${n\vee m}$ is the componentwise maximum. The following lemma is basically Lemma 2.4 from \cite{re2} for MGDS with (DC).
\begin{lemma} \label{lematogerms}
Let $(X,T)$ be a MGDS with (DC)\\
(i) A partial homeomorphism $S$ belongs to $\mathcal{G}(X,T)$ if and only if it is locally of the form $(T_{|U}^{m})^{-1}T_{|V}^{n}$ where
$m,n\in \mathbb{N}^{r}$, $U$ is an open set on which $T^{m}$ is injective and $V$ is an open set on which $T^{n}$ is injective.\\
(ii)Let $a\in X$. Suppose that $(T_{|U}^{m})^{-1}T_{|V}^{n}$ and $(T_{|W}^{p})^{-1}T_{|Y}^{q}$ are two partial homeomorphisms as in (i)
having $a$ in their domains and $(T_{|U}^{m})^{-1}T_{|V}^{n}a=(T_{|W}^{p})^{-1}T_{|Y}^{q}a$. If $m-n=p-q$, then $(T_{|U}^{m})^{-1}T_{|V}^{n}$ and $(T_{|W}^{p})^{-1}T_{|Y}^{q}$
have the same germ at $a$.
\end{lemma}
\begin{proof}
It is clear that $(T_{|U}^{m})^{-1}T_{|V}^{n}\in \mathcal{G}(X,T)$ and, since $\mathcal{G}(X,T)$ is full, it is still true for a partial homeomorphism locally
of this form. The inverse of $(T_{|U}^{m})^{-1}T_{|V}^{n}$ is $(T_{|V}^{n})^{-1}T_{|U}^{m}$ which belongs to $\mathcal{G}(X,T)$.
It remains to show that the product of $(T_{|U}^{m})^{-1}T_{|V}^{n}$ and $(T_{|W}^{p})^{-1}T_{|Y}^{q}$ is locally of the same form.
When $r=1$ this is the alternative $n\geq p$ or $n\leq p$ given in the proof of Lemma 2.4 of \cite{re2}. In our setting this alternative does not work since $\mathbb{N}^{r}$ is not totally ordered.
For this reason we need the condition (DC).
Let $x\in Y$ such that $(T_{|W}^{p})^{-1}T_{|Y}^{q}x=y\in V$. We can suppose that this happens in a neighborhood of $x$ and so we assume that it is true on $Y$.
Then with $z=(T_{|U}^{m})^{-1}T_{|V}^{n}y$ we have $T^{q}x=T^{p}y$ and $T^{n}y=T^{m}z$. From condition (DC) we have $y\in\textrm{dom}(T^{p\vee n})$ and
$T^{q+p\vee n-p}x=T^{p\vee n}y=T^{m+p\vee n-n}z$ for any $x\in Y$ so that $(T_{|U}^{m})^{-1}T_{|V}^{n}(T_{|W}^{p})^{-1}T_{|Y}^{q}$ is locally of the form
$(T_{|Z}^{m+p\vee n-n})^{-1}T_{|Z'}^{q+p\vee n-p}$.\\
(ii)Taking neighborhoods of $a$ and $(T_{|U}^{m})^{-1}T_{|V}^{n}a$ we may assume that $W=U$ and $Y=V$, $T^{p}(U)=T^{q}(V)$ and $T^{m\vee p}$ is injective on $U$.
For $x\in V$, $(T_{|U}^{m})^{-1}T_{|V}^{n}x=y\in U$, $(T_{|U}^{p})^{-1}T_{|V}^{q}x=y'\in U$ we have $T^{m}y=T^{n}x$ and $T^{p}y'=T^{q}x$.
As $U\subset\textrm{dom}(T^{m})\cap\textrm{dom}(T^{p})\subset\textrm{dom}(T^{m\vee p})$ we have
$T^{m\vee p}y=T^{n+m\vee p-m}x=T^{q+m\vee p-p}x=T^{m\vee p}y'$ and therefore $y=y'$ so
$(T_{|U}^{m})^{-1}T_{|V}^{n}=(T_{|U}^{p})^{-1}T_{|V}^{q}$.
\end{proof}
Having defined a pseudogroup, we denote by Germ($X,T$) the groupoid of germs of $\mathcal{G}(X,T)$.
 
Following \cite{re2} Definition.2.5 we consider another groupoid, the semidirect product groupoid
(simply replacing $\mathbb{Z}$ with $\mathbb{Z}^{r}$):
\begin{defn}\label{semidirect} Let $(X,T)$ be a MGDS with (DC). Its semidirect groupoid is
$$
G(X,T)=\{(x,m-n,y);m,n\in \mathbb{N}^{r},x\in dom(T^{m}), y\in dom(T^{n}), T^{m}x=T^{n}y\}
$$
with the groupoid structure induced by the product structure of the trivial groupoid
$X\times X$ and of the group $\mathbb{Z}^{r}$.
The topology is defined by the basic open sets
$$\mathcal{U}(U;m,n;V)=\{(x,m-n,y): (x,y)\in U\times V, T^{m}x=T^{n}y\} $$
where $U$ (respectively $V$) is an open subset of the domain of $T^{m}$ (respectively $T^{n}$) on which $T^{m}$
(respectively $T^{n}$) is injective. 
\end{defn}
The family of given subsets is indeed a basis for a topology since\\
$\mathcal{U}(U;m,n;V)\cap\mathcal{U}(U';m',n';V')\supset \mathcal{U}(U\cap U';m\vee m',n\vee n';V\cap V')$.
Thus, $\gamma=(x,z,y)$ and $\eta=(x',z',y')$ in $G(X,T)$ are composable if and only if $y=x'$ and then
$\gamma \eta=(x,z+z',y')$. The range and domain are $r(x,z,y)=x$ and $d(x,z,y)=y$. This is a groupoid indeed since $y\in \textup{dom}(T^{n\vee m'})$ so
$T^{m+n\vee m'-n}x=T^{n'+n\vee m'-m'}y'$. In the absence of the condition (DC), $G(X,T)$ may not be a groupoid. Consider, for example, $(X,L,R)$ as in example \ref{example}(iii). The condition (DC) is not satisfied since $\textup{dom}(L)\cap \textup{dom}(R)$ is the set of nonvoid words, while dom($LR$) is the set of words with length greater or equal to 2. Let $|\cdot|$ be the word
length on $X$ and $\Omega$ the empty word. For $x,y\in X$ we have
$\gamma=(x,(\vert x\vert,-\vert y \vert),y)\in G(X,T)$ and $\eta=(y,(\vert y \vert, 0),\Omega)\in G(X,T)$ since $T^{(\vert x \vert,0)}x=L^{\vert x \vert}x=\Omega=R^{\vert y \vert}y=T^{(0,\vert y \vert)}$
and $T^{(\vert y \vert,0)}y=L^{\vert y \vert}y=\Omega=T^{0}\Omega$ but
$\gamma\eta=(x,(\vert x\vert+\vert y \vert, -\vert y \vert),\Omega)\notin G(X,T)$ since $x\notin \textup{dom}(T^{n})$ for any $n\in \mathbb{N}^{2}$ with
$n_{1}>\vert x\vert$.

We assume again that $(X,T)$ have (DC). According to (ii) of the previous lemma, there is a map $\pi$ from $G(X,T)$ onto Germ($X,T$) which sends
$(x,m-n,y)$ into the germ $[x,(T_{|U}^{m})^{-1}T_{|V}^{n},y]$ where $U$ is an open neighborhood of $x$ on which $T^{m}$ is injective
and $V$ is an open neighborhood of $y$ on which $T^{n}$ is injective. This map is continuous and is a groupoid homomorphism.
It is an isomorphism when $(X,T)$ is essentially free.
\begin{defn}(Conform Definition.2.6 \cite{re2})\label{esfree}
We shall say that a MGDS $(X,T)$ is essentially free if for every pair of distinct $m,n\in \mathbb{N}^{r}$, there is no open set on which $T^{n}$ and $T^{m}$ agree.
\end{defn}
\begin{lemma}(Conform Lemma 2.7 \cite{re2})
Let $(X,T)$ be an essentially free MGDS with (DC). Then:\\
(i)If $(T_{|U}^{m})^{-1}T_{|V}^{n}$ and $(T_{|W}^{p})^{-1}T_{|Y}^{q}$, where $m,n,p,q\in \mathbb{N}$ and $U,V,W,Y$ are open sets such that
$T_{|U}^{m},T_{|V}^{n},T_{|W}^{p},T_{|Y}^{q}$ are injective and have the same germ at $a$, then $m-n=p-q$.\\
(ii)The map $c:\textrm{Germ}(X,T)\to\mathbb{Z}^{r}$ such that $c[(T_{|U}^{m})^{-1}T_{|V}^{n}x,(T_{|U}^{m})^{-1}T_{|V}^{n},x]=m-n$ is a continuous homomorphism.
\end{lemma}
\begin{proof}
(i) By assumption, we have $T^{m}y=T^{n}x$ and $T^{p}y=T^{q}x$ for $x$ and $y$ in neighborhoods of $a$ respectively $b=(T_{|U}^{m})^{-1}T_{|V}^{n}a$.
Then on these neighborhoods we have $T^{n+m\vee p-m}x=T^{m\vee p}y=T^{q+m\vee p-p}x$. The essential freeness implies that $n+m\vee p-m=q+m\vee p-p$ so
$n-m=q-p$.\\
(ii) We have seen in the proof of the previous lemma that the product of $(T_{|U}^{m})^{-1}T_{|V}^{n}$ and $(T_{|W}^{p})^{-1}T_{|Y}^{q}$ is locally of the form
$(T_{|Z}^{m+p\vee n-n})^{-1}T_{|Z'}^{q+p\vee n-p}$ and that the inverse of $(T_{|U}^{m})^{-1}T_{|V}^{n}$ is $(T_{|V}^{n})^{-1}T_{|U}^{m}$.
This shows that $c$ is a homomorphism and by construction it is locally constant.
\end{proof}
\begin{prop} \label{thesame}
Let $(X,T)$ be a MGDS with (DC). Then $(X,T)$ is essentially free if and only if the above surjection
$\pi:G(X,T)\to\textrm{Germ}(X,T)$ is an isomorphism.
\end{prop}
\begin{proof}
Word with word as in Proposition 2.8 \cite{re2}.
\end{proof}
This homomorphism induces a morphism of algebras even if $(X,T)$ is not essentially free. We construct it in the lemma.
\begin{lemma} \label{togerms}
Let $\pi:G_{1}\to G_{2}$ a surjective morphism between two r-discrete groupoids with the property that if $s(\pi(x))=r(\pi(y))$ then $s(x)=r(y)$. Then the correspondence $C_{c}(G_{1})\ni f\to \tilde{\pi}(f)\in C_{c}(G_{2})$, $\tilde{\pi}(f)(x)=\sum_{\pi(x')=x}f(x')$, is a *-homomorphism between the topological algebras $C_{c}(G_{1})$ and $C_{c}(G_{2})$. Composing it with the regular representation of $C_{c}(G_{2})$ we obtain a bounded representation of $C_{c}(G_{1})$ therefore a morphism $\tilde{\pi}:C^{*}(G_{1})\to C^{*}_{r}(G_{2})$  
\end{lemma}
\begin{proof}
The condition in the statement is $(\pi\times \pi)^{-1}(G_{2}^{2})=G_{1}^{2}$. This means that composable pairs lifts only to composable pairs. In particular, the restriction $\pi^{0}$ of $\pi$ to the unit space must be a bijection. Moreover, $\pi$ satisfies this condition if $\pi^{0}$ is a bijection which identifies the equivalence relations associated to $G_{1}$ and $G_{2}$. In particular if $\pi(u)=r(x)$ then $\{x';\pi(x')=x\}\subset G_{1}^{u}$. Indeed, if $\pi(x')=\pi(y')=x$ then $\pi(y'^{-1})$ and $\pi(x')$ are composable, so $y'^{-1}$ and $x'$ are composable, hence $r(x')=r(y')$. This means that $\pi(G_{1}^{u})=G_{2}^{\pi^{0}(u)}$. If $f\in C_{c}(G_{1})$ then $\{x':\pi(x')=x\}\subset \textrm{supp}(f)\cap G_{1}^{u}$ which is finite. Therefore the sum which defines $\tilde{\pi}(f)$ is finite. We have $\textrm{supp}(\tilde{\pi}(f))\subset\pi(\textrm{supp}(f)$ so $\tilde{\pi}(f)\in C_{c}(G_{2})$.

We compute
$$\tilde{\pi}(f\star g)(x)=\sum_{\pi(x')=x}f\star g (x')=\sum_{\pi(y'z')=x}f(y')g(z')$$
$$\tilde{\pi}(f)\star\tilde{\pi}(g)(x)=\sum_{yz=x}\tilde{\pi}(f)(y)\tilde{\pi}(g)(z)$$
Since $\pi$ is surjective we have 
$$\tilde{\pi}(f)\star\tilde{\pi}(g)(x)=\sum_{\pi(y')\pi(z')=x}f(y')g(z')$$
The first sum runs over $\{(y',z'): \pi(y')\pi(z')=x\}$ and the second over sum runs over $\{(y',z'): \pi(y'z')=x\}$. These two sets are equal by the assumption on $\pi$ so $\tilde{\pi}$ is an algebraic morphism

If $f_{n}\in C_{c}(G_{1})$, $\textrm{supp}(f_{n})\subset K$, $f_{n}\rightarrow^{u} f$ on K, then $\tilde{\pi}(f_{n})\rightarrow^{u} \tilde{\pi}(f)$ since the cardinal of the set $\{x':\pi(x')=x\}$ is bounded when $x$ runs over a compact set.

Finally, one has 
\begin{eqnarray*}
\|\tilde{\pi}(f)\|_{I}\\
&=&
sup_{u\in G_{2}^{0}}\sum_{r(x)=u}|\tilde{\pi}(f)(x)|\\
&=&
sup_{u\in G_{2}^{0}}\sum_{r(x)=u}|\sum_{\pi(x')=x}f(x')|\\
&\leq&
sup_{u\in G_{2}^{0}}\sum_{r(x)=u, \pi(x')=x}|f(x')|
\end{eqnarray*}

Since $\pi(G_{1}^{u})=G_{2}^{\pi^{0}(u)}$ and $\pi^{0}$ is a bijection, the last sum is 
$$sup_{u\in G_{1}^{0}}\sum_{r(x)=u}|f(x)|=\|f\|_{I}\geq \|\tilde{\pi}(f)\|_{I}$$
Therefore $\|f\|_{I}\geq\|\tilde{\pi}(f)\|_{B(L^{2}(G_{2}))}$ since the regular representation is bounded.
\end{proof} 

\begin{prop}\label{germs}
The Lemma \ref{togerms} holds for $G_{1}=G(X,T)$, $G_{2}=Germ(X,T)$ and $\pi$ the morphism of Lemma \ref{lematogerms}.
\end{prop}
\begin{proof}
$\pi$ is continuous and surjective. $\pi(x,z,y)$ and $\pi(x',z',y')$ are composable if and only if $y=x'$ which is the condition for $(x,z,y)$ and $(x',z',y')$ to be composable.
\end{proof}

We assume from now on that $X$ is Hausdorff, second countable and locally compact. Then $G(X,T)$ becomes a Hausdorff, locally compact \'etale groupoid.
In the next theorem we prove the amenability of $G(X,T)$. For the proof we need some notations. 
The motivation comes from the example \ref{example}. 
For $1\leq j\leq r$ let $X_{j}=\bigcap_{n\in \mathbb{N}} \textrm{dom}(T_{j}^{n})$ and for a subset $J\subset \{1,\ldots,r\}$ ($J$ may be void) we denote by
$X_{J}=\cap_{j\in J}X_{j}\bigcap\cap_{j\notin J}(X\backslash X_{j})$. Clearly $X_{J}\cap X_{J'}=\Phi$ for $J\neq J'$ and for $x\in X$ we have
$x\in X_{J_{x}}$ where $J_{x}=\{j: x\in X_{j}\}$ so the subsets $X_{J}$ provide a partition of $X$. When $r=1$ we get the partition in the end of the proof of
Proposition 2.9 (i) \cite{re2}.
For $x\in X$ let
$$\sigma(x)=\textrm{sup}\{n\in \mathbb{N}^{r}:x\in \textrm{dom}(T^{n})\}\in \overline{\mathbb{N}}^{r}$$
For the rank one case $\sigma(x)$ may be considered as the exit time of $T$, the first time when $x$
escapes the domain of $T$. In general $\sigma(x)_{j}$ can be thought of as the exit time of $T_{j}$ so we call $\sigma(x)$ the \textit{exit time of T}.
The crucial assumption (DC) ensures the existence of this supremum.
Then $X_{J}=\{x\in X:\sigma(x)_{j}=\infty \textrm{ for }j\in J \textrm{ and } \sigma(x)_{j} \textrm{ finite for }j\notin J\}$.
For $n\in \overline{\mathbb{N}}^{r}$ and $J=\{j:n_{j}=\infty\}$
we have $\sigma^{-1}(n)=\cap_{j\in J}X_{j}\bigcap \cap_{j\notin J}\textrm{dom}(T^{n_{j}})\backslash\textrm{dom}(T^{n_{j}+e_{j}})\subset X_{J}$
which is Borel and analytic so that a preimage by $\sigma$ is analytic. Note that from the condition (DC) we have
$\cap_{j\notin J}\textrm{dom}(T^{n})\backslash\textrm{dom}(T^{n+e_{j}})=\cap_{j\notin J}\textrm{dom}(T^{n_{j}}_{j})\backslash\textrm{dom}(T^{n_{j}+1}_{j})$.
Denote by $\mathbb{Z}^{J}$ the subgroup of $\mathbb{Z}^{r}$ given by the inclusion
$\mathbb{Z}^{|J|}\ni z\mapsto \sum_{j\in J}z_{j}e_{j}\in \mathbb{Z}^{r}$ and, for $z\in \mathbb{Z}^{r}$,
let $z_{J}=\sum_{j\in J}z_{j}e_{j}\in \mathbb{Z}^{J}$. We use a similar notation for $\mathbb{N}^{J}$ and $n_{J}$.
\begin{lemma}  \label{middle}
(i) For $x,y\in X$ and $n\in \mathbb{N}^{r}$ we have $\sigma(T^{n}x)=\sigma(x)-n$.\\
(ii)For $(x,z,y)\in G(X,T)\vert_{X_{J}}$ we have $z_{J^{c}}=\sigma(x)_{J^{c}}-\sigma(y)_{J^{c}}$ where $J^{c}=\{1,\ldots,r\}\backslash J$.
\end{lemma}
\begin{proof}
(i) By definition of $\sigma$ we have $\sigma(T^{n}x)=\textrm{sup}\{m\in \mathbb{N}^{r}:x\in \textrm{dom}(T^{m+n})\}=
\textrm{sup}\{m-n\in \mathbb{N}^{r}:m\geq n, x\in \textrm{dom}(T^{m})\}=\textrm{sup}\{m\in \mathbb{N}^{r}:m\geq n, x\in \textrm{dom}(T^{m})\}-n=
\sigma(x)-n$.\\
(ii) By definition of $G(X,T)$ there exist $n,m\in \mathbb{N}^{r}$ such that $T^{n}x=T^{m}y$ so by (i) we have $\sigma(T^{n}x)=\sigma(x)-n=\sigma(T^{m}y)=\sigma(y)-m$.
We can substract the finite $J^{c}$ coordinates to get $z_{J^{c}}=n_{J^{c}}-m_{J^{c}}=\sigma(x)_{J^{c}}-\sigma(y)_{J^{c}}$.
\end{proof}

The following lemma is very likely folklore:
\begin{lemma}\label{decomposition}
Let $G$ be a Borel groupoid, $(X_{j})_{j\in J}$, $J$ a finite set, a partition of $G^{0}$ by invariant Borel sets.
Then $G$ is measurewise amenable if and only if $G\vert_{X_{j}}$ is measurewise amenable.
\end{lemma}
\begin{proof}
This follows directly form the definition of amenability (\cite{ar} Definition. 3.2.8).
Precisely, we denote by $\sqcup$ the disjoint union of Borel set with the obvious Borel structure.
We know that $G^{0}=\sqcup X_{j}$, $G=\sqcup G\vert_{X_{j}}$ so
$L^{\infty}(G)=\oplus L^{\infty}(G\vert_{X_{j}})$ and $L^{\infty}(G^{0})=\oplus L^{\infty}(X_{j})$. Therefore a mean
$m:L^{\infty}(G)\to L^{\infty}(G^{0})$ is invariant if and only if the restrictions $m_{j}:L^{\infty}(G\vert_{X_{j}})\to L^{\infty}(X_{j})$ are invariant.
Conversely, we can piece $m_{j}$ together to define $m=\oplus m_{j}$ an invariant mean $L^{\infty}(G)\to L^{\infty}(G^{0})$.
\end{proof}
\begin{theorem} \label{amenable}(Conform Prop.2.9 \cite{re2}) Let $(X,T)$ be a MGDS with (DC).
Then,\\
i) $G(X,T)$ is amenable;\\
ii) the full and reduced $C^{*}$-algebras coincide;\\
iii) the $C^{*}$-algebra $C^{*}(X,T)=C^{*}(G(X,T))$ is nuclear;
\end{theorem}
\begin{proof}(i) We will check measurewise amenability (according to \cite{ar}, 3.3.7, it is equivalent to topological amenability for \'etale groupoids).
Each $X_{J}$ is an invariant Borel set for $G(X,T)$. By Lemma \ref{decomposition} it is enough to prove that the reduction of $G(X,T)$ on $X_{J}$ is amenable.
By the definition of $X_{J}$, the homomorphism $c_{J}:G(X,T)\vert_{X_{J}}\to \mathbb{Z}^{J}$ defined by $c_{J}(x,z,y)=z_{J}$ is strongly surjective in the sense given in
\cite{ar} Definition.5.3.7. We shall show that $R_{J}=c_{J}^{-1}(0)$ is an amenable equivalence relation.
Once this proven, we can now apply a result on the amenability of an extension (\cite{ar}, 5.3.14) to conclude that $G(X,T)\vert_{X_{J}}$ is amenable, hence
$G(X,T)$ amenable.
We have
$$R_{J}=\{(x,m-n,y):x,y\in X_{J}, \exists n,m\in \mathbb{N}^{r}, \textrm{ with }m_{J}=n_{J},$$
$$\;\;\;\;\;x\in \textrm{dom}(T^{n}), y\in \textrm{dom}(T^{m}), T^{m}(x)=T^{n}(y)\}.$$
If $(x,z,y)\in R_{J}$
we have $z_{J}=0$ and $z_{J^{c}}=\sigma(x)_{J^{c}}-\sigma(y)_{J^{c}}$ (Conform Lemma \ref{middle}(ii))
so $z$ depends only on $x,y$ and
$R_{J}\subset X_{J}\times X_{J}$. We leave $z$ out when we have such an element. We show that $R_{J}$ is an equivalence relation.
If $(x,y),(y,z)\in R_{J}$ then we can find $n,n',m,m'\in \mathbb{N}^{r}$
such that $n_{J}=n'_{J}$, $T^{n}x=T^{n'}y$, $m_{J}=m'_{J}$, $T^{m}y=T^{m'}z$. As we have seen before, we can assume $n_{J^{c}}=\sigma(x)_{J^{c}}$,
$n'_{J^{c}}=\sigma(y)_{J^{c}}$, $m_{J^{c}}=\sigma(y)_{J^{c}}$, $m'_{J^{c}}=\sigma(z)_{J^{c}}$. (DC) gives again
$T^{n_{J}\vee m_{J}+\sigma(x)_{J^{c}}}x=T^{n_{J}\vee m_{J}+\sigma(y)_{J^{c}}}y=T^{n_{J}\vee m_{J}+\sigma(z)_{J^{c}}}z$
and then $(x,z)\in R_{J}$ which proves that $R_{J}$ is an equivalence relation.

We shall show that $R_{J}$ is an inductive limit of amenable equivalence relations. For $N\in \mathbb{N}^{J}$ let
$$R^{N}_{J}=\{(x,y)\in R_{J}: \exists n,m\in \mathbb{N}^{r} \textrm{ such that }n_{J}=m_{J}\leq N, T^{n}(x)=T^{m}(y)\}$$
In view of Lemma \ref{middle}(ii), we can choose $n$, $m$ such that $n_{J^{c}}=\sigma(x)_{J^{c}}$ and $m_{J^{c}}=\sigma(y)_{J^{c}}$ so $R^{N}_{J}$
can be described as
$$R^{N}_{J}=\{(x,y)\in R_{J}:\exists n\in \mathbb{N}^{J}, n\leq N, \textrm{ such that }T^{n+\sigma(x)_{J^{c}}}(x)=T^{n+\sigma(y)_{J^{c}}}(y)\}.$$
$R_{J}^{N}$ is an equivalence relation on $X_{J}$ since we have seen before that $T^{n_{J}\vee p_{J}+\sigma(x)_{J^{c}}}x=T^{n_{J}\vee p_{J}+\sigma(z)_{J^{c}}}z$
for $(x,y,m,n),(y,z,p,q)$ defining two elements in $R_{J}^{N}$ as above. $R_{J}=\cup_{N\in \mathbb{N}^{J}}R_{J}^{N}$, $(R_{J}^{N})^{0}=X_{J}=(R_{J})^{0}$ and
$R_{J}^{N}=R_{J}^{N+1}$. We shall show that $R_{J}^{N}$ is a proper equivalence relation. This ensures that $R_{J}$ is the inductive limit of $R_{J}^{N}$
in the sense of \cite{ar} Chapter 5.3.f. Then Prop. 5.3.37 of \cite{ar} gives the amenability of $R_{J}$.

Since $R_{J}^{N}$ is a discrete equivalence relation  we have to show that the space $X_{J}/R_{J}^{N}$
is analytic (Conform \cite{ar} Example 2.1.4(2)). According to \cite{ber} Corollary 4.12, this follows from the countable separability of $X_{J}/R_{J}^{N}$.
Since $\sigma$ is a Borel map, we can find a sequence $Y_{i}$ of Borel subsets of $X_{J}$ such that $Y_{i}\subset \{x\in X_{J}:\sigma(x)_{J^{c}}=k\}$
for some $k\in \mathbb{N}^{J^{c}}$ and
the restriction of $T^{n}$ to $Y_{i}$ is injective. We can suppose furthermore that $(Y_{i})_{i}$ is separating for $X_{J}$ and is closed under finite intersections.
The saturation of a Borel set $A\subset X_{J}$ is $[A]=\cup_{n\in \mathbb{N}^{J}, n\leq N}(T^{n+\sigma(\cdot)_{J^{c}}})^{-1}T^{n+\sigma(\cdot)_{J^{c}}}(A)$ which is a Borel set.
We use here the notation $T^{n+\sigma_{J^{c}}(\cdot)}$ for the Borel map $x\mapsto T^{n+\sigma_{J^{c}}(x)}(x)$.
Let $(x,y)\notin R_{J}^{N}, x\in Y_{i}$. We prove that we can choose $Y_{j}$ such that $y\notin [Y_{j}]$ so $\pi(Y_{j})$ separates $\pi(x)$ and $\pi(y)$ where
$\pi$ is the projection from $X_{J}$ to $X_{J}/R_{J}^{N}$. If $y\notin Y_{i}$ we are done. If $y\in [Y_{i}]$, there exists $y'\in Y_{i}$ such that $(y,y')\in R_{J}^{N}$.
Since the family $(Y_{i})_{i}$ separates $X_{J}$ and is closed under finite intersections, we can find another
set $Y_{j}\subset Y_{i}$ which separates $x$ and $y'$, that is $x\in Y_{j}$ and $y'\notin Y_{j}$. If $y\notin [Y_{j}]$ we are back to the case $y\notin [Y_{i}]$.
If $y\in [Y_{j}]$ then $(y,y'')\in R_{J}^{N}$ for some
$y''\in Y_{j}\subset Y_{i}$ so $(y',y'')\in R_{J}^{N}$. This means that there exists $n\in \mathbb{N}^{J}$, $n\leq N$ such that
$T^{n+\sigma(y')_{J^{c}}}(y')=T^{n+\sigma(y'')_{J^{c}}}(y'')$. Since $y',y''\in Y_{i}$ we have $\sigma(y')_{J^{c}}=\sigma(y'')_{J^{c}}$.
$T^{n+\sigma(\cdot)_{J^{c}}}$ is injective on $Y_{i}$ so $y'=y''$.
This contradicts the choice of $Y_{j}$ ($y'\notin Y_{j}$). Therefore, $y\notin [Y_{j}]$.
\end{proof}

\begin{cor}\label{coretogerms}
There is a morphism $\tilde{\pi}:C_{r}(G(X,T))\to C_{r}(Germ(X,T))$ induced by the canonical morphism $\pi:G(X,T)\to Germ(X,T)$. 
\end{cor}
\begin{proof}
Conform Proposition \ref{germs} there is a morphism $C^{*}_{r}(G(X,T))=C^{*}(G(X,T))\to C^{*}_{r}(Germ(X,T))$
\end{proof}

\section{Duality in bivariant theories} \label{duality}
$K$-theory and $K$-homology are particular cases of the Kasparov groups.
The product in $KK$-theory provides the analogue of the cup and cap products in algebraic topology. Therefore $KK$-theory and other bivariant theories are the framework
for notions of duality. $KK$-theory is suitable when geometric classes are available $Ext$-theory was used in \cite{kapu,em1,em2}. The novelty in \cite{pz} is to use $Ext^{r}$ but a technical condition of semisplitness leads to $E$-theory. We give first a general notion of duality analogous with the Spanier-Whitehead duality in algebraic topology. Then we give a condition which implies this duality. It appeared for the first time in \cite{ka} Chapter 4, Th. 6, but was highlighted by Connes in \cite{co} Chapter VI.4.$\beta$. After that we recall the construction of long exact sequences starting from an algebra and a tuple of ideals. 

Let $F$ be any of the bivariant theories: $KK$-theory, $E$-theory, Ext-theory or KExt-theory.
For $x\in F(A_{1}\otimes \ldots \otimes A_{n},B_{1}\otimes\ldots \otimes B_{m})$ we define the flip maps $\sigma_{ij}(x)$ and $\sigma^{ij}(x)$
the induced maps on $F$-groups by the flip homomorphisms of $C^{*}$-algebras
$A_{1}\otimes \ldots \otimes A_{i}\otimes \ldots \otimes A_{j} \otimes A_{n} \rightarrow A_{1}\otimes \ldots \otimes A_{j}\otimes\ldots \otimes A_{i}\otimes\ldots \otimes A_{n}$ respectively
$B_{1}\otimes \ldots \otimes B_{i}\otimes\ldots \otimes B_{j}\otimes\ldots \otimes B_{n} \rightarrow B_{1}\otimes \ldots \otimes B_{j}\otimes\ldots \otimes B_{i}\otimes\ldots \otimes B_{n}$.

\begin{defn}(Spanier-Whitehead duality, \cite{kapu}, Definition. 2.2)
Let $A,B$ be separable $C^{*}$-algebras, $\Delta \in F^{r}(A\otimes B, \mathbb{C})$, $\delta\in F^{r}(\mathbb{C}, A\otimes B)$.
We say that $A$ and $B$ are Spanier-Whitehead $r$-dual in $F$-theory if the maps
$\Delta_{i}:K_{i}(A)\rightarrow K^{i+r}(B)$, $\Delta_{i}(x)=x\otimes_{A}\Delta$ and
$\sigma^{12}(\Delta)_{i}:K_{i}(B)\rightarrow K^{i+r}(A)$, $\sigma^{12}(\Delta)_{i}(x)=x\otimes_{B}\Delta$ are isomorphisms with inverses
$\delta_{i}:K^{i+r}(B)\rightarrow K_{i}(A)$, $\delta_{i}(y)=\delta \otimes_{B} y$ respectively
$\sigma_{12}(\delta)_{i}:K^{i+r}(A)\rightarrow K_{i}(B)$, $\sigma_{12}(\delta)_{i}(y)=\sigma_{12}(\delta)\otimes_{A} y$.
\end{defn}
The duality classes, $\Delta$ and $\delta$, are the $K$-homology and the $K$-theory classes.
\begin{theorem}(\cite{em2} Theorem 11) \label{connesnotion}
Let $A,B,\Delta, \delta$ be $C^{*}$-algebras and classes as above such that the following equations hold:
$$\delta\otimes_{B}\sigma^{12}(\Delta)=[1_{A}]$$
$$\sigma_{12}(\delta)\otimes_{A}\Delta=[(-1)^{i}1_{B}]$$
Then the maps $\Delta_{i}$ and $\delta_{i}$ defined in the above definition are isomorphisms.
\end{theorem}
The proof can be found in \cite{em2}. We shall call such algebras simply $r$-dual.
The sign in the second condition comes from the change of sign under flip maps $\sigma$.
The sign is not given in \cite{ka} and \cite{co}.
We do not know of any example of two dual $C^{*}$-algebras such that the assumptions of
Th. \ref{connesnotion} are not fulfilled. If $B$ is the same as $A$ we say that $A$ is a Poincar\'e duality algebra.
The dual is not unique since if $A$ has a dual $B$ and $A$ is $KK$-equivalent to $A'$ then $A'$ is also dual with $B$.

Bivariant theories work well when the algebra in the first argument is separable. For $KK$-theory this is forced by the Kasparov technical theorem. Since in the definition of
duality both algebras appear in the first argument we are forced to work with separable algebras. For a separable algebra the $K$-theory group is countable. For this reason
not every algebra has a dual. For example $M_{2^{\infty}}$, the CAR algebra, does not have a dual. If there was a dual $B$ for $M_{2^{\infty}}$ we would have
$K_{i}(B)=K^{1}(M_{2^{\infty}})$ where $i$ is 0 or 1. But $K^{1}(M_{2^{\infty}})$ is $\mathbb{Z}_{(2)}/\mathbb{Z}$ (\cite{had}), where $\mathbb{Z}_{(2)}$ is the group of 2-adic numbers,
which is uncountable. Therefore $B$ cannot be separable.

In \cite{pz} the $K$-homology class was given as an $E$-theory class associated to a long exact sequences. We recall the construction of long exact sequences starting from an algebra and a tuple of ideals. 

Let $\mathcal{E}$ be an arbitrary $C^{*}$-algebra and  $J_{1}, \ldots, J_{r}$ be $r$ ideals in it.
We shall describe next a procedure of defining an r-fold exact sequence
starting with $B=\bigcap_{i=1}^{r}J_{i}$ and ending with $A=\raisebox{3pt}{$\mathcal{E}$}\Big/ \raisebox{-4pt}{$J_{1}+\ldots+J_{r}$}$. The motivation comes from the Zekri's Yonneda product (\cite{ze}) $\epsilon_{1}\otimes_{\mathbb{C}}\epsilon_{2}$  of two 1-fold extensions $\epsilon_{1}\in \textup{Ext}(A_{1},B_{1})$ and $\epsilon_{2}\in \textup{Ext}(A_{2},B_{2})$.
Let
$$0\rightarrow B_{1}\rightarrow E_{1}\rightarrow A_{1}\rightarrow 0$$
$$0\rightarrow B_{2}\rightarrow E_{2}\rightarrow A_{2} \rightarrow 0$$
be the extensions $\epsilon_{1}$ and $\epsilon_{2}$. Forgetting for the moment the troubles caused by tensor products, the product $\epsilon_{1}\otimes_{\mathbb{C}}\epsilon_{2}$ is computed by tensoring
$\epsilon_{1}$ with $B_{2}$ on the right, $\epsilon_{2}$ with $A_{1}$ on the left and then splicing:
$$0\rightarrow B_{1}\otimes B_{2}\rightarrow E_{1}\otimes B_{2}\rightarrow A_{1}\otimes E_{2}\rightarrow A_{1}\otimes A_{2}\rightarrow 0$$
Define $\mathcal{E}=E_{1}\otimes E_{2}$,
$J_{1}=B_{1}\otimes E_{2}$, $J_{2}=E_{1}\otimes B_{2}$. We have
$$B_{1}\otimes B_{2}=J_{1}\cap J_{2}=J_{1}J_{2}$$
$$E_{1}\otimes B_{2}=J_{2}$$
$$A_{1}\otimes E_{2}=\raisebox{3pt}{$\mathcal{E}$}/\raisebox{-3pt}{$J_{1}$}$$
$$A_{1}\otimes A_{2}=\raisebox{3pt}{$\mathcal{E}$}/\raisebox{-3pt}{$J_{1}+J_{2}$}$$
so the exact sequence is
$$0\rightarrow J_{1}\cap J_{2}\rightarrow J_{2}\rightarrow \raisebox{3pt}{$\mathcal{E}$}/\raisebox{-3pt}{$J_{1}$}\rightarrow \raisebox{3pt}{$\mathcal{E}$}/\raisebox{-3pt}{$J_{1}+J_{2}$}\rightarrow 0$$
For a nonempty subset $S \subs \{1, \ldots , r\}$ let us define
$$
J^{S} = \bigcap_{j \in S} J_{j} \; \text{ and } \; J_{S} = \sum_{j \in S} J_{j}.
$$
We shall use the abbreviation $\{k,k+1,\ldots,l\} = \overline{k,l}$. Define
\begin{eqnarray*}
\mathcal{E}_{0}&=&J_{1}\cap \ldots \cap J_{r}= J_{\overline{1,r}} \\
\mathcal{E}_{1}&=&J_{2}\cap \ldots \cap J_{r}= J_{\overline{2,r}} \\
\mathcal{E}_{k}&=& \raisebox{3pt}{$J_{k+1}\cap \ldots \cap J_{r}$} \Big/ \raisebox{-4pt}{$(J_{1} + \ldots + J_{k-1}) \cap J_{k+1} \cap \ldots \cap J_{r}$} = \raisebox{3pt}{$J^{\overline{k+1,r}}$} \Big/ \raisebox{-4pt}{$J_{\overline{1,k-1}} \cap J^{\overline{k+1,r}}$},
\end{eqnarray*}
for any $k\in\{2,\ldots r-1\}$ and
\begin{eqnarray*}
\mathcal{E}_{r}&=& \raisebox{3pt}{$\mathcal{E}$} \Big/ \raisebox{-4pt}{$ J_{1} + \ldots + J_{r-1}$} = \raisebox{3pt}{$\Ee$} \Big/ \raisebox{-4pt}{$J^{\overline{1,r-1}}$}\\
\mathcal{E}_{r+1}&=& \raisebox{3pt}{$\mathcal{E}$} \Big/ \raisebox{-4pt}{$J_{1} + \ldots + J_{r}$} = \raisebox{3pt}{$\Ee$} \Big/ \raisebox{-4pt}{$J^{\overline{1,r}}$}.
\end{eqnarray*}
$\mathcal{E}_{0} \subseteq \mathcal{E}_{1}$ so we can define $i_{0}: \mathcal{E}_{0}\hookrightarrow \mathcal{E}_{1}$.
Since $\mathcal{E}_{1} \subseteq J_{3}\cap \ldots \cap J_{r}$, there is a map
$i_{1}:\mathcal{E}_{1}\rightarrow \mathcal{E}_{2}$ given by the inclusion composed with
the quotient map.
For $k\in\{2,\ldots,r-2\}$ we have
$$
J^{\overline{k+1,r}}=J_{k+1}\cap \ldots \cap J_{r}\subseteq J_{k+2}\cap \ldots \cap J_{r}=J^{\overline{k+2,r}}
$$
and
$$
J_{\overline{1,k-1}} \cap J^{\overline{k+1,r}} \subseteq J_{\overline{1,k}} \cap J^{\overline{k+2,r}}
$$
so that we can again define a map $i_{k}:\mathcal{E}_{k}\rightarrow\mathcal{E}_{k+1}$ as the bottom line
of the diagram
\begin{displaymath}
\xymatrix
{J^{\overline{k+1,r}}\ar@{^{(}->}[r] \ar[d]^{\pi_{k}} &J^{\overline{k+2,r}}\ar[d]^{\pi_{k+1}}\\
 \mathcal{E}_{k} \ar[r]^{i_{k}} & \mathcal{E}_{k+1}
}
\end{displaymath}
where the vertical arrows are projections.
Using the isomorphism $\raisebox{3pt}{$J$}\Big/\raisebox{-3pt}{$J\cap I$}\simeq \raisebox{3pt}{$J+I$}\Big/\raisebox{-3pt}{$I$}$ we write
$$\Ee_{r-1} =
\raisebox{3pt}{$J_r$} \Big/ \raisebox{-4pt}{$J_{\overline{1,r-2}}\cap J_{r}$}=
\raisebox{3pt}{$J_1+ \dots + J_{r-2} +J_r$} \Big/ \raisebox{-4pt}{$J_1+\dots + J_{r-2}$}$$
Therefore, there is also a natural homomorphism
$i_{r-1}:\mathcal{E}_{r-1}\rightarrow\mathcal{E}_{r}$ since $J_1+ \dots + J_{r-2} +J_r \subs \Ee$ and
$J_1+\dots + J_{r-2} \subs J_1+\dots + J_{r-1}$. Finally
$i_{r}:\mathcal{E}_{r}\rightarrow\mathcal{E}_{r+1}$ is defined since $J_1+\dots + J_{r-1} \subs J_1+ \dots + J_r$.
\begin{theorem} (\cite{pz}, Prop. 3.1)\label{constructionideals}
The r-fold sequence
$$
\xymatrix{0 \ar[r]&B\ar[r]^{i_0}&\mathcal{E}_{1}\ar[r]^{i_1}& \cdots \ar[r]^{i_{r-1}}&\mathcal{E}_{r} \ar[r]^{{i_r}}& A \ar[r]& 0}
$$
is exact.
\end{theorem}

Such sequences define $Ext^{r}$ classes provided they are semisplit. In \cite{pz} this technical condition was avoided using $E$-theory, that is thinking of such a sequence as an $E^{r}$-theory class. However,
as a consequence of the fact that $\textup{Ext}^{1}(A,B)=\overline{\textup{Ext}}^{1}(A,B)$ is a group and the Yonneda product is bilinear, we can prove that $\textup{Ext}^{r}(A,B)$ is a group.
Indeed, we can decompose any $\epsilon \in \overline{\textup{Ext}^{r}}(A,B)$ as a Yonneda product $\gamma(\epsilon_{1},\epsilon_{2})$ with
$\epsilon_{1}\in \overline{\textup{Ext}}^{1}(A,A_{1})$. Since $\textup{Ext}^{1}(A,A_{1})$ is a group, we can find an extension, $\epsilon_{1}'\in \textup{Ext}^{1}(A,A_{1})$
such that $\epsilon_{1}\oplus \epsilon_{1}'$ is a null (i.e. split) extension in $\textup{ext}^{1}(A,A_{1})$.
The inverse of $\epsilon$ is now $\gamma(\epsilon_{1}',\epsilon_{2})$. As a consequence, if $A$ is nuclear then $\epsilon_{1}$ is always semisplit so
$\gamma(\epsilon_{1}',\epsilon_{2})$ provides an inverse for $\epsilon$ in $\overline{\textup{Ext}}^{r}(A,B)$. Therefore $\overline{\textup{Ext}}^{r}(A,B)$ is a group which
contains $\textup{Ext}^{r}(A,B)$. We conjecture that they are equal if $A$ is nuclear. More generally we can define a group if in the semisplitness condition we only
require that $\epsilon_{r}$, the rightmost 1-fold exact sequence, is semisplit. We believe that this is equal to $\textup{Ext}^{r}(A,B)$. In any case it factors through $KK$ as a consequence of abstract characterizations (\cite{cohi}).

For example if $G$ is a groupoid and $X_{1},... X_{r}$ are $r$ open invariant subsets of $G^{0}$ we get a tuple of ideals of $C_{r}^{*}(G)$ $(C_{r}^{*}(G\vert_{X_{1}}),C_{r}^{*}(G\vert_{X_{2}}),...C_{r}^{*}(G\vert_{X_{r}}))$. The algebras $\mathcal{E}_{k}$ corresponds to $C_{r}^{*}(G\vert_{Y_{k}})$ where
$$Y_{0}=X_{1}\cap ...\cap X_{r}$$
$$Y_{1}=X_{2}\cap ...\cap X_{r}$$
$$Y_{k}=(X_{k+1}\cap ...\cap X_{r})\setminus (X_{1}\cup ...\cup X_{k-1})$$
$$Y_{r}=X\setminus(X_{1}\cup ...\cup X_{r-1})$$
$$Y_{r+1}=X\setminus(X_{1}\cup ...\cup X_{r})$$
In particular, if $(X,T)$ is a MGDS with (DC) condition and $X_{j}$ are the set used in the proof of the amenability of $G(X,T)$ we have 
$$Y_{k}=\{x\in X: \sigma(x)_{j}<\infty \textrm{ for } j\geq k+1 \textrm{ and } \sigma(x)_{j}=\infty \textrm{ for } j\leq k-1\}$$

One can improve the result of \cite{pz} for any higher rank graph $\Lambda$ with the following finiteness condition:
$$
\textrm{(F): The set } \{\lambda:\sigma(\lambda)=n\} \textrm{ is nonvoid and finite for any } n\in\mathbb{N}^{r}
$$
such that $0<\#\Lambda_{n}(v)<\infty$ and $0<\#\Lambda^{n}(v)<\infty$ where $n\in \mathbb{N}^{r}$, $v\in \Lambda^{0}$. We shall give up the aperiodicity condition and replace the Toeplitz algebra $\mathcal{E}$ with the subalgebra $\tilde{\mathcal{E}}$ of $\mathcal{E}\oplus C^{*}(\Lambda)\otimes C^{*}(\Lambda^{op})$ generated by $\tilde{L}_{\lambda}=L_{\lambda}\oplus (s_{\lambda}\otimes 1)$ and $\tilde{R}_{\lambda}=R_{\lambda}\oplus (1\otimes t_{\lambda})$. We refer to $\mathcal{E}$ as the first summand of $\tilde{\mathcal{E}}$ and to $C^{*}(\Lambda)\otimes C^{*}(\Lambda^{op})$ as to the second summand. This trick will give rise to a kind of pull-back for long exact sequences. The constructions are the same as in \cite{pz} with $O_{\Lambda}$ replaced by $C^{*}(\Lambda)$ and $O_{\Lambda^{op}}$ replaced by $C^{*}(\Lambda^{op})$. In this way we avoid aperiodicity conditions. 

The closed linear span of a set $S$ in a normed linear space is denoted by $[S]$ and the projection onto
a closed subspace $\Ll$ of a Hilbert space $\Hh$ by $P_{\Ll}$.
For any $j \in \{1 \ldots r\}$ and $a \in \Lambda^{o}$ we define
$$
P_{a}=P_{[ \Omega, \delta_{\lambda} \mid  o(\lambda)=a]} \; , \;
P_{a}^{j}=P_{[\Omega, \delta_{\lambda} \mid  o(\lambda)=a \textrm{ and } \sigma(\lambda)_{j}=0 ]} \; ,
$$
$$
Q_{a}=P_{[ \Omega, \delta_{\lambda} \mid  t(\lambda)=a]} \; , \;
Q_{a}^{j}=P_{[\Omega, \delta_{\lambda} \mid  t(\lambda)=a \textrm{ and } \sigma(\lambda)_{j}=0]} \; ,
$$
$$
P^j = P_{[\Omega , \delta_{\lambda} \mid \sigma (\lambda)_{j}=0]}.
$$
Define $L_{\Omega} =1=R_{\Omega}$ and for $a \in \Lambda^{o}$ let $L_{a}$ be the projection
$P_{a}$ and $R_{a}$ the projection $Q_{a}$. 
It is easy to check that for any
$\mu\in \Lambda^{*},\  j \in \{1 , \ldots , r\}$ and $a \in \Lambda^{o}$ we have
$$
L_{\mu}^{*}L_{\mu}=P_{s(\mu)},\  P_{a}=L_{a}=
\sum_{t(\lambda)=a; \sigma(\lambda)=e_{j}}\quad L_{\lambda}L_{\lambda}^{*}+P_{a}^{j},
$$
$$
R_{\mu}^{*}R_{\mu}=Q_{t(\mu)},\quad  Q_{a}=R_{a}=
\sum_{t(\lambda)=a; \sigma(\lambda)=e_{j}} R_{\lambda}R_{\lambda}^{*}+Q_{a}^{j}.
$$
Note also that
$$
1- \sum_{\sigma(\lambda)=e_j} L_{\lambda} L_{\lambda}^{*} = P^j = 1- \sum_{\sigma(\lambda)=e_j} R_{\lambda} R_{\lambda}^{*}
$$
and for all $k \in \mathbb{N}^{r}$ we have
$$
\sum_{\sigma(\lambda)=k} L_{\lambda} L_{\lambda}^{*} = \sum_{\sigma(\lambda)=k} R_{\lambda} R_{\lambda}^{*} = P_{[\delta_{\mu} \mid \sigma (\mu) \geq k]}.
$$
We have similar relations in the algebra $\tilde{\mathcal{E}}$ if we replace $L$ with $\tilde{L}$ and $R$ with $\tilde{R}$:
$$\tilde{L}_{a}=
\sum_{t(\lambda)=a; \sigma(\lambda)=e_{j}}\quad \tilde{L}_{\lambda}\tilde{L}_{\lambda}^{*}+P_{a}^{j},
$$
$$
\tilde{R}_{a}=
\sum_{t(\lambda)=a; \sigma(\lambda)=e_{j}} \tilde{R}_{\lambda}\tilde{R}_{\lambda}^{*}+Q_{a}^{j}.
$$
$$
1- \sum_{\sigma(\lambda)=e_j} \tilde{L}_{\lambda} \tilde{L}_{\lambda}^{*} = P^j = 1- \sum_{\sigma(\lambda)=e_j} \tilde{R}_{\lambda} \tilde{R}_{\lambda}^{*}
$$
As in \cite{pz} the Toeplitz algebra $\mathcal{E}$ contains a tuple of $r$ ideals $(J_{1},...J_{r})$, $J_{j}=\langle P^{j} \rangle$ the closed two-sided ideals generated by $P^j$
in $\tilde{\mathcal{E}}$. Note that the ideals $J_{k}$ are in the first summand.

As in \cite{pz}, Lemma 2.1 we have $\bigcap_{j=1}^r J_{j} =\mathbb{K}(\mathsf{F})$ and as in \cite{pz} Theorem 2.2 there is a morphism $C^{*}(\Lambda)\otimes C^{*}(\Lambda^{op})
\simeq \tilde{\mathcal{E}}/J_{ \{ 1,\ldots,r \}}$ given by
$s_{\lambda}\otimes 1 \mapsto \widehat{\tilde{L}_{\lambda}}$ and
$1\otimes t_{\lambda} \mapsto \widehat{\tilde{R}_{\lambda}}$. The pull back construction, more exactly the second summand, is used to prove the injectivity of the above isomorphism which used the uniqueness of the Cuntz-Krieger relations associated to the graph. Indeed, since $J_{ \{ 1,\ldots,r \}}\subset \pi_{1}(\mathcal{E})$ we have the following diagram:
$$
\xymatrix
{
C^{*}(\Lambda)\otimes C^{*}(\Lambda^{op})\ar[r]\ar[dr]^{id}&\tilde{\mathcal{E}}/J_{ \{ 1,\ldots,r \} }\ar[d]\\
&C^{*}(\Lambda)\otimes C^{*}(\Lambda^{op})
}
$$
where the vertical arrow is the projection onto the second summand: $\mathcal{E}/J_{ \{ 1,\ldots,r \}}\ni \hat{x}\oplus y\mapsto y\in C^{*}(\Lambda)\otimes C^{*}(\Lambda^{op})$.
From the construction in section \ref{duality} we get a long exact sequence
$$0\rightarrow \mathbb{K}(\mathsf{F})\rightarrow \mathcal{E}_{1}\rightarrow\ldots\rightarrow \mathcal{E}_{r}\rightarrow C^{*}(\Lambda)\otimes C^{*}(\Lambda^{op})\rightarrow 0$$
therefore an $E$-theory class $\Delta \in E^{r}(C^{*}(\Lambda)\otimes C^{*}(\Lambda^{op}), \mathbb{C})=K^{r}(C^{*}(\Lambda)\otimes C^{*}(\Lambda^{op}))$. To compute the product $\delta\otimes_{C^{*}(\Lambda^{op})}\sigma_{12}(\Delta)=\tau_{C^{*}(\Lambda)}(\delta)\otimes \tau^{C^{*}(\Lambda)}(\sigma_{12}(\Delta))$, which is a class in
$E^{0}(\mathcal{S}^{\otimes r}\otimes C^{*}(\Lambda),\mathcal{S}^{\otimes r}\otimes C^{*}(\Lambda))$, we find $[\Theta]\otimes\tau_{C^{*}(\Lambda)}(\delta)\otimes \tau^{C^{*}(\Lambda)}(\sigma_{12}(\Delta))
=[\tau_{C^{*}(\Lambda)}(\delta)\circ\Theta]\otimes \tau^{C^{*}(\Lambda)}(\sigma_{12}(\Delta))$ where $\Theta: \mathcal{S}^{\otimes r}\otimes C^{*}(\Lambda)\to \mathcal{S}^{\otimes r}\otimes C^{*}(\Lambda)$. This is defined by  restricting the morphism $\Theta:C(\mathbb{T}^{r})\otimes C^{*}(\Lambda)\longrightarrow C(\mathbb{T}^{r})\otimes C^{*}(\Lambda)$,
$\Theta(f)(z)=\gamma_{z}(f(z))$, $\gamma$ the gauge action, to  $\mathcal{S}^{\otimes r}\otimes C^{*}(\Lambda)$. This restriction is in turn homotopic to the identity. Our product is a pull-back (\cite{ght} Prop 5.8), that is the top row of the following diagram:
$$
\xymatrix
{
0 \ar[r]& C^{*}(\Lambda)\otimes\mathcal{E}_{0} \ar[r] \ar@{=}[dd] & \mathcal{E}_{1}' \ar[r] & \dots \ar[r] & \mathcal{E}_{r}' \ar[r] &\\
        &                                                  &                         &              &                        & \\
0 \ar[r]& C^{*}(\Lambda)\otimes\mathcal{E}_{0} \ar[r]         & C^{*}(\Lambda)\otimes\mathcal{E}_{1} \ar[r]  & \dots \ar[r] & C^{*}(\Lambda)\otimes\mathcal{E}_{r} \ar[r] &
}
$$
$$
\;\;\;\;\;\;\;\;
\xymatrix{
&\ar[r]& C(\mathbb{T}^{r})\otimes C^{*}(\Lambda) \ar[r] \ar[d]^{\Theta}      & 0 \\
&&C(\mathbb{T}^{r})\otimes C^{*}(\Lambda) \ar[d]^{\delta \otimes id}      \\
&\ar[r]&C^{*}(\Lambda)\otimes C^{*}(\Lambda^{op})\otimes C^{*}(\Lambda) \ar[r]      & 0
}
$$

To show that it gives the same element with 
$\tau_{C^{*}(\Lambda)}(\mathcal{T}^{\otimes r})\in E^{r}(\mathcal{S}^{\otimes r}\otimes C^{*}(\Lambda),C^{*}(\Lambda))$ (which is $1_{C^{*}(\Lambda)}\in E(C^{*}(\Lambda),C^{*}(\Lambda))$by Bott periodicity) we shall identify a subsequence which gives $\tau_{C^{*}(\Lambda)}(\mathcal{T}^{\otimes r})$. This is done by identifying in $C^{*}(\Lambda)\otimes\tilde{\mathcal{E}}$ an image $\mathcal{E}'$ of the algebra $\mathcal{T}^{\otimes r}\otimes C^{*}(\Lambda)$ which preserves the canonical tuple of ideals of $\mathcal{T}^{\otimes r}\otimes C^{*}(\Lambda)$. Therefore it gives a commutative diagram

$$
\xymatrix{
0\ar[r] & C^{*}(\Lambda)\otimes\mathcal{E}_{0} \ar[r] & \mathcal{E}_{1}' \ar[r]  & \cdots \ar[r] & \mathcal{E}_{r}'\longrightarrow  \\
0\ar[r] & \mathcal{T}_{0}\otimes C^{*}(\Lambda) \ar[r] \ar[u]& \mathcal{T}_{1}\otimes C^{*}(\Lambda) \ar[r] \ar[u]& \cdots \ar[r] & \mathcal{T}_{r}\otimes C^{*}(\Lambda) \longrightarrow \ar[u] 
}
$$

$$
\;\;\;\;\;
\xymatrix{
&\ar[r]&C(\mathbb{T}^{r})\otimes C^{*}(\Lambda)  \ar[r] \ar@{=}[d] & 0 \\
&\ar[r]&C(\mathbb{T}^{r})\otimes C^{*}(\Lambda) \ar[r] &0
}
$$
In addition, the leftmost vertical arrow gives an $E$-theoretic equivalence since the full corner $C^{*}(\Lambda)\otimes P_{\Omega}\subset C^{*}(\Lambda)\otimes\mathcal{E}_{0}$ corresponds to the full corner $P_{0}\otimes C^{*}(\Lambda)\subset \mathcal{T}_{0}\otimes C^{*}(\Lambda)$ are . 
To construct the algebra $\mathcal{E}'$ we define 
$$W_{j}=\sum_{\sigma(\lambda)=e_{j}}s^{*}_{\lambda}\otimes \tilde{R}_{\lambda}$$
$$V_{\lambda}=\big[ W^{\sigma(\lambda)} \big]^{*}(1\otimes \tilde{L}_{\lambda})$$
where for a tuple of $r$ commuting operators $x=(x_{1},...x_{r})$ and $k\in \mathbb{N}^{r}$ we define $x^{k}=x_{1}^{k_{1}}..x_{r}^{k_{r}}$. One can prove as in \cite{pz} Proposition 6.4 that $C^{*}(\Lambda)\simeq C^{*}(V_{\lambda} \mid \lambda\in \Lambda^{*})$, the isomorphism being given by
$s_{\lambda}\mapsto V_{\lambda}$. The pull-back construction replaces the uniqueness of $O_{\Lambda}$ invoked in \cite{pz} Proposition 6.4 and gives the injectivity. 
The isometries $(W_{k})_{k}$ commute with each other and with $V_{\lambda}$'s. So we have a morphism $\mathcal{T}^{\otimes r}\otimes C^{*}(\Lambda) \to C^{*}(W_{k}, V_{\lambda})$ which is the identity on the full corner $P_{\Omega}\otimes C^{*}(\Lambda)$. All these statements follow exactly as in \cite{pz}.

\section{A Groupoid Picture of the duality for higher rank graph algebra}

In this section we give this duality a groupoid interpretation. 
Let us start with the $K$-theory class. More precisely, from the groupoid picture of the algebra $C^{*}(\Lambda)\otimes C^{*}(\Lambda)$ the partial unitaries $w_{k}$ are two-sided shifts on the double infinite paths space. Indeed, conform \cite{pz}, Proposition 4.1
$$w_{k}w_{k}^{*}=w_{k}^{*}w_{k}=\sum_{a\in \Lambda^{o}}s_{a}\otimes t_{a}$$
In the groupoid $G(\Lambda^{\infty},T)\times G((\Lambda^{op})^{\infty},T^{op})$, this projection corresponds to the set $1_{Z}$ where 
$$Z=\{(x,y)\in\Lambda^{\infty}\times(\Lambda^{op})^{\infty};s(x)=s(y)\}$$
The isometry $w_{k}$ is given by the bisection
$$\sum_{\sigma(\lambda)=e_{k}}1_{\{((\lambda x,y),(e_{k},-e_{k}),(x,\lambda y))\} }$$
$$
\xy
{\ar (0,0)*{}; (20,0)*{}};
{\ar (20,-20)*{}; (20,0)*{}};
(20,3)*{t(\lambda)};
{\ar (25,20)*{}; (25,0)*{}};
{\ar (45,0)*{}; (25,0)*{}};
{\ar@{.>} (30,20)*{}; (30,0)*{}};
(27,-3)*{\lambda};
(52,0)*{(e_{k},-e_{k})};
(10,-10)*{y};
(37,10)*{x};
 {\ar (60,0)*{}; (80,0)*{}};
 {\ar (80,-20)*{}; (80,0)*{}};
 {\ar@{.>} (75,-20)*{}; (75,0)*{}};
 (77,3)*{\lambda};
 (65,-10)*{y};
 (90,10)*{x};
 {\ar (85,20)*{}; (85,0)*{}};
 (85,-3)*{s(\lambda)};
 {\ar (105,0)*{}; (85,0)*{}};
\endxy
$$
Identifying the set $Z$ with the set of two-sided infinite words we can say that $w_{k}$ is the two sided shift in the direction $k$. The transformation group of the $n$ shifts on $Z$, $Z\rtimes \mathbb{Z}^{n}$ is an open subgroupoid of $G(\Lambda^{\infty},T)\times G((\Lambda^{op})^{\infty},T^{op})$. 
Note that here one can see clearly the necessity of the condition (F) on $\Lambda$. Without it the space $Z$ may not be locally compact or the two-sided shifts on the double infinite path space may not be well defined on $Z$ (when the graph has sinks or sources).
The morphism $C(\mathbb{T}^{r})\rightarrow C^{*}(\Lambda)\otimes C^{*}(\Lambda^{op})$, $z_{k}\mapsto w_{k}$ is given by the inclusion of $C^{*}(\mathbb{Z}^{r})\to C(Z)\ltimes \mathbb{Z}^{r}$.
The twist morphism $\Theta$ is given by the topological isomorphism of groupoids
$$\mathbb{Z}^{r}\times G(\Lambda^{\infty},T)\ni (z,x)\mapsto (z+b(x),x)\in\mathbb{Z}^{r}\times G(\Lambda^{\infty},T)$$
where $b$ is the cocycle which gives the gauge action: $b(x,z,y)=z$. 

Identifying $\mathcal{S}^{\otimes r}$ with $C^{*}(\mathbb{R}^{r})$ the restriction of this morphism to $\mathcal{S}^{\otimes r}\otimes C^{*}(\Lambda)$ is given by 
$$\mathbb{R}^{r}\times G(\Lambda^{\infty},T)\ni (t,x)\mapsto (t+b(x),x)\in\mathbb{R}^{r}\times G(\Lambda^{\infty},T).$$
The homotopy between this morphism and the identity is obtained by multiplying $b(x)$ with a parameter $s\in [0,1]$.

To give a groupoid description of the $K$-homology class it would be ideal to have a conceptual  groupoid description (like semidirect product) of the two-sided Toeplitz algebra $\mathcal{E}$. We have not been able to do it but we can still use a groupoid of germs. Let  $X_{l}$, $X_{r}$, $X$ be the Gelfand spectrum of the commutative unital algebras $\mathcal{E}_{l}\cap l^{\infty}(\overline{\Lambda})$, $\mathcal{E}_{r}\cap l^{\infty}(\overline{\Lambda})$ respectively $\mathcal{E}\cap l^{\infty}(\overline{\Lambda})$ where $l^{\infty}(\Lambda)\subset\mathbb{L}(l^{2}(\overline{\Lambda}))$ is the algebra of diagonal operators. There is an isomorphism from $L_{\lambda}^{*}L_{\lambda}C(X)L_{\lambda}^{*}L_{\lambda}$ to $L_{\lambda}L_{\lambda}^{*}C(X)L_{\lambda}L_{\lambda}^{*}$ given by $f\mapsto L_{\lambda}f L_{\lambda}^{*}$. Therefore the isometries $\{L_{\lambda}:\lambda\in\Lambda\}$ give rise by Gelfand duality to partial homeomorphisms of $X_{l}$, so they generate a pseudogroup $\mathcal{G}_{l}$. Similarly,  $\{R_{\lambda}:\lambda\in\Lambda\}$ give rise to partial homeomorphisms of $X_{r}$ and they generate a pseudogroup $\mathcal{G}_{r}$. The set of partial isometries $\{L_{\lambda}, R_{\lambda}:\lambda\in\Lambda\}$ gives also partial homeomorphisms $\{l_{\lambda}, r_{\lambda}:\lambda\in\Lambda\}$ of $X$ which generate a pseudogroup $\mathcal{G}$. We denote by $G_{l}$, $G_{r}$, $G$ the corresponding groupoids of germs. 
It is important to have in mind that the maps $l_{\lambda}$ and $r_{\lambda}$ extend the maps $s(\lambda)\Lambda\ni x \mapsto \lambda x\in \lambda\Lambda$ respectively $\Lambda t(\lambda)\ni x\mapsto x\lambda\in \Lambda\lambda$ so we shall freely use the notation $\lambda x=l_{\lambda}x$ and $y\lambda=r_{\lambda}y$ whenever $x \in dom(l_{\lambda})$ and $y\in dom(r_{\lambda})$. Here our convention is that $\Omega x=x$ for any $x\in X$.

$X_{l}$ is easily described since the algebra $\mathcal{E}_{l}\cap l^{\infty}(\overline{\Lambda})$ is generated by the projections $L_{\lambda}L_{\lambda}^{*}$. It is the space appearing in Example \ref{example}(i) \cite{fmy}. $X_{r}$ is the same space when we replace $\Lambda$ with $\Lambda^{op}$. However, the space $X$ seems to be much more complicated. The difficulty appears from the projections $P_{j}L_{\lambda}^{*}R_{\mu}R_{\mu}^{*}L_{\mu}P_{j}$ with $\sigma(\lambda)=\sigma(\mu)=e_{j}$ which are not in the algebra generated by the projections $L_{\lambda}L_{\lambda}^{*}$ and $R_{\lambda}R_{\lambda}^{*}$
This projection is the domain of the partial isometry $R_{\mu}^{*}L_{\mu}P_{j}$.
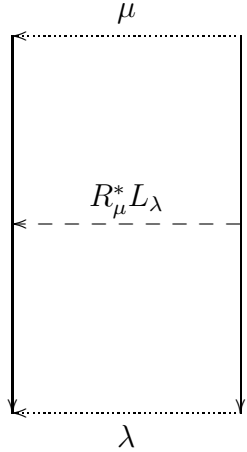
$$
\xy
 {\ar@{.>} (30,50)*{}; (0,50)*{}};
 {\ar (30,50)*{}; (30,0)*{}};
 {\ar (0,50)*{}; (0,0)*{}};
 {\ar@{.>} (30,0)*{}; (0,0)*{}};
 (15,-3)*{\lambda};
 (15,53)*{\mu};
 (15,28)*{R_{\mu}^{*}L_{\lambda}};
 {\ar@{-->}(30,25)*{};(0,25)*{}};
\endxy
$$
We do not know something like shift equivalence to describe this isometry.
This does not happen in the rank one case since left and right creations commute up to compact operators. In the rank one case the set $X$ is the disjoint union $\overline{\Lambda}\bigsqcup (\Lambda^{\infty}\times(\Lambda^{op})^{\infty})$.

\begin{lemma}
\begin{enumerate}
 \item 
The set $\overline{\Lambda}$ is an open dense discrete subset of $X$ ($X$ is a compactification of $\overline{\Lambda}$)
\item
The algebra $\mathcal{E}$ is the image of the canonical representation $\pi$ of $C_{r}^{*}(G)$ on $l^{2}(\overline{\Lambda})=\mathsf{F}$
\end{enumerate}
\end{lemma}
\begin{proof}
 \begin{enumerate}
\item  $\mathcal{E}$ contains the ideal of compact operators $\mathbb{K}(\mathsf{F})$ so $\mathcal{E}\cap l^{\infty}(\overline{\Lambda})$ contains the diagonal algebra $\mathbb{K}(\mathsf{F})\cap l^{\infty}(\overline{\Lambda})=C_{0}(\overline{\Lambda})$ as an essential ideal. Therefore $X$ contains the discrete set $\overline{\Lambda}$ as a dense open subset. $X$ is compact since it is the spectrum of an unital algebra.
\item
The set $\overline{\Lambda}$ is invariant with respect to the partial homeomorphisms $l_{\lambda}, r_{\lambda}$ so it is invariant for the groupoid $G$. It is also open so it gives rise to an ideal of $C_{r}^{*}(G)$. Since the reduction $G\vert_{\overline{\Lambda}}$ is the transitive groupoid, this ideal is the ideal of compact operators. 
We have a representation $\pi$ of $C_{r}^{*}(G)$ on $l^{2}(\overline{\Lambda})$ induced by the canonical representation of the ideal $\mathbb{K}(l^{2}(\overline{\Lambda}))$: $\pi(T)\delta_{\lambda}=\pi(TP_{[\delta_{\lambda}]})\delta_{\lambda}$, $T\in C_{r}^{*}(G)$. The image $\pi(C_{r}^{*}(G))$ in $\mathbb{L}(l^{2}(\overline{\Lambda}))$ is exactly the algebra $\mathcal{E}$. Indeed, a basis for the topology of the groupoid $G$ is given by equivalence classes (modulo passing to germs) of $(Y,w)$ where $1_Y$ is a projection in $C(X)=\mathcal{E}\cap l^{\infty}(\overline{\Lambda})$ and $v$ is a finite word in $l_{\lambda}, l_{\lambda}^{-1}, r_{\lambda}, r_{\lambda}^{-1}$ such that $v^{-1}$ is a homeomorphism on $Y$. Therefore $\pi(1_{(Y,w)})=1_{Y}w\in \mathcal{E}$, where $w$ is the word in $L_{\lambda}, L_{\lambda}^{*}, R_{\lambda}, R_{\lambda}^{*}$ obtained by replacing $l$ and $r$ with $L$ and $R$. Therefore, the image of the representation $\pi$ is contained in $\mathcal{E}$ and it is clear that $L_{\lambda}, R_{\lambda}$ are in this image so $\pi(C_{r}^{*}(G))=\mathcal{E}$.
\end{enumerate}
\end{proof}

We want now to identify the morphism $\mathcal{T}^{\otimes r}\otimes C^{*}(\Lambda)\to \mathcal{E}\otimes C^{*}(\Lambda)$, $S_{i}\to W_{i}$ and $s_{\lambda}\to V_{\lambda}$. In the following proposition we give the main properties of $X$. The continuous map $\sigma$ defined on $X_{l}$ and $X_{r}$ is the extended shape given in Example \ref{example}(i).

\begin{prop} \label{mainprop}
\begin{enumerate}
\item 
There are canonical continuous surjective maps $p_{l}$, $p_{r}$ and $\sigma$ such that the following diagram is commutative
$$
\xymatrix
{
&X_{l}\ar[dr]^{\sigma}&\\
X\ar[ur]^{p_{l}}\ar[dr]_{p_{r}}&&\overline{\mathbb{N}}^{r}\\
&X_{r}\ar[ur]_{\sigma}&
}
$$
where $p_{l}(\lambda)=p_{r}(\lambda)=\lambda$ for any $\lambda\in\overline{\Lambda}$ and $\sigma$ extends the shape on $\overline{\Lambda}$.
\item
The source and terminal maps $s$ and $t$ on $\Lambda$ extends to continuous maps on $X$.
\item
$X$ has the following factorization property: if $x\in X,\sigma(x)\geq k+p$, $k,p\in \mathbb{N}^{r}$ then there exists unique $\lambda, \mu\in \Lambda$ with $\sigma(\lambda)=k$, $\sigma(\mu)=p$ and $y\in X$ with $\sigma(y)=\sigma(x)-k-p$ such that $x=\lambda y\mu$.
\item
$J_{k}=\pi(C_{r}^{*}(G\vert_{X_{k}}))$ where $X_{k}=\{x\in X:\sigma(x)_{k}<\infty\}$
\item
Let $Z=\{(x,y)\in X\times \Lambda^{\infty}: s(x)=t(y)\}$. Then the map
$$\phi=(\sigma,\psi):( \Lambda\times\Lambda^{\infty})\cap Z\ni (x,y)\mapsto (\sigma(x),xy)\in \overline{\mathbb{N}}^{n}\times \Lambda^{\infty}$$
extends to a continuous surjective map $\phi:Z\to \overline{\mathbb{N}}^{n}\times \Lambda^{\infty}$
\end{enumerate}
\end{prop}
\begin{proof}
 \begin{enumerate}
 \item 
We have the inclusion of algebras $\mathcal{E}_{l}\cap l^{\infty}(\overline{\Lambda})\subset\mathcal{E}\cap l^{\infty}(\overline{\Lambda})$, that is $C(X_{l})\subset C(X)$,  with identity on $C_{0}(\Lambda)$ hence there is a continuous surjection $X\to X_{l}$ which is the identity on the set $\overline{\Lambda}$. Similarly, there is a continuous surjection $X\to X_{r}$ which is the identity on the set $\overline{\Lambda}$. The diagram is commutative because $\sigma p_{l}=\sigma p_{r}$ on $\overline{\Lambda}$ which is a dense set. We denote by $\sigma$ the shape extended to $X_{l}$, $X_{r}$ or $X$.
\item
$X$ is a disjoint union of $dom(l_{a})$ so we can define $s(x)=a$ if $x\in dom(l_{a})$ and $t(x)=a$ if $x\in dom(r_{a})$. $s$ and $t$ are continuous since $s^{-1}(\{a\})=dom(l_{a})$ and $t^{-1}(\{a\})=dom(r_{a})$.
\item
We have a disjoint union $$X=\bigcup_{\sigma(\lambda)=k}dom(l_{\lambda})=\bigcup_{\sigma(\lambda)=k}dom(r_{\lambda})$$
Therefore there is one and only one $\lambda$ with $\sigma(\lambda)=k$ such that $x\in dom(l_{\lambda})$. We can take $x'=l_{\lambda}^{-1}(x)$ and then $x=\lambda x'$. Moreover $\sigma(x')=\sigma(x)-k$. This can be checked on the set $\{\lambda\mu:\mu\in\overline{\Lambda}\}$ which is a dense set in $dom(l_{\lambda}^{-1})=ran(l_{\lambda})$. We repeat now this reason for right factorization of $x'$.
\item
The projection $P_{k}$ corresponds to the open set $\{x\in X:\sigma(x)_{k}=0\}$ which generates the open $G$-invariant set $\{x\in X:\sigma(x)_{k}<\infty\}$.
\item
We have to show that the maps $(\Lambda\times\Lambda^{\infty})\cap Z\ni (x,y)\mapsto \sigma(x)\in \overline{\mathbb{N}}^{r}$ and $(\Lambda\times\Lambda^{\infty})\cap Z\ni (x,y)\mapsto xy\in \Lambda^{\infty}$ extend to continuous maps. The first claim is proven in $(i)$ above. To prove the second let $(\lambda_{n},y_{n})$ be a convergent sequence in $Z$. We have to prove that $(\lambda_{n}y_{n})(0,m)$ is eventually constant. Since $y_{n}$ is convergent in $\Lambda^{\infty}$ we know that $y_{n}(0,m)$ is eventually constant so we can suppose it is $\mu$. We know that the sequence $(\lambda_{n}\mu)_{n}$ converges in $X$ since $R_{\mu}$ is continuous so it converges also in $X_{l}$. This implies that $(\lambda_{n}\mu)(0,m)=(\lambda_{n}y_{n})(0,m)$ is eventually constant.
For the surjectivity of $\phi$ we use the density of the set $\mathbb{N}^{r}\times \Lambda^{\infty}$ in  $\overline{\mathbb{N}}^{r}\times \Lambda^{\infty}$. This dense set is in the image of $\phi$ which is closed since Z is compact.
\end{enumerate}
\end{proof}

From the factorization in (iii) of the proposition above, we can define commuting left and right semigroups of shifts $(L,R)$ on  $X$ with $dom(L_{k})=dom(R_{k})=\{x\in X:\sigma(x)\geq k\}$
$$L_{k}x=x', \;\;R_{k}x=x'' \textrm{ where } x=\lambda x'=x''\mu \textrm{ with }\sigma(\lambda)=\sigma(\mu)=k$$
These shifts are continuous since their restrictions to $dom(l_{\lambda})$, respectively $dom(r_{\lambda})$, are $l_{\lambda}^{-1}$ and $r_{\lambda}^{-1}$. 

\begin{prop}
\begin{enumerate}
 \item 
The maps 
$$T_{m}:\{(x,y)\in Z:x=x'x'',\sigma(x'')= m\}\ni (x,y)\mapsto (x',x''y'')\in Z$$ 
$$V_{m}:Z\to Z, \;\;V_{m}(x,y)=(L_{m}(x(y(0,m))),L_{m}(y))$$
define a MGDS $(T,V)$ on $Z$ which satisfies the condition (DC).
\item
The \textit{exit time} map $\sigma$ defined in section \ref{cuntzlike} is the same with the map $\sigma$ in the previews Proposition.
\end{enumerate}
\end{prop}
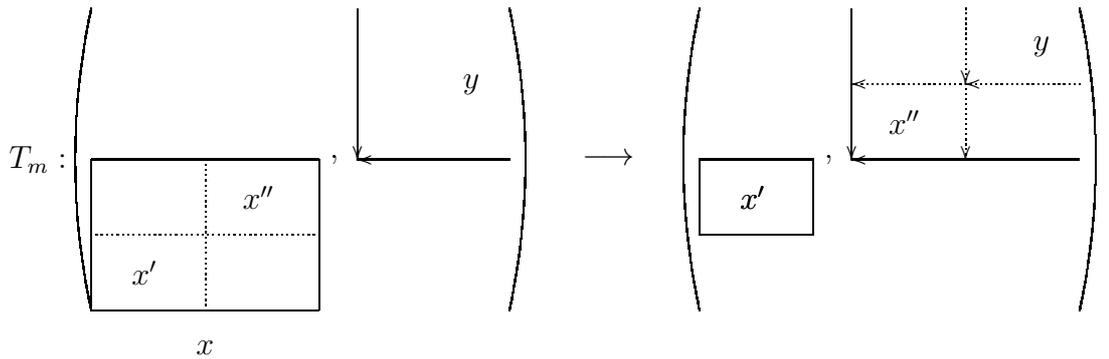
$$
\xy
{\ar@/^/@{-} (0,-20)*{}; (0,20)*{}};
{\ar@/_/@{-} (55,-20)*{}; (55,20)*{}};
{(0,0)*{}; (30,0)*{} **\dir{-}};
{(30,0)*{}; (30,-20)*{} **\dir{-}};
{(30,-20)*{}; (0,-20)*{} **\dir{-}};
{(0,-20)*{}; (0,0)*{} **\dir{-}};
{(0,-10)*{}; (30,-10)*{} **\dir{.}};
{(15,0)*{}; (15,-20)*{} **\dir{.}};
{\ar (55,0)*{}; (35,0)*{}};
{\ar (35,20)*{}; (35,0)*{}};
(32,0)*{,};
(15,-25)*{x};
(7,-15)*{x'};
(22,-5)*{x''};
(50,10)*{y};
(68,0)*{\longrightarrow};
(-7,0)*{T_{m}:};
{\ar@/^/@{-} (80,-20)*{}; (80,20)*{}};
{(80,0)*{}; (95,0)*{} **\dir{-}};
{(95,0)*{}; (95,-10)*{} **\dir{-}};
{(95,-10)*{}; (80,-10)*{} **\dir{-}};
{(80,0)*{}; (80,-10)*{} **\dir{-}};
(87,-5)*{x'};
(97,0)*{,};
{\ar@/_/@{-} (130,-20)*{}; (130,20)*{}};
{\ar (100,20)*{}; (100,0)*{}};
{\ar (130,0)*{}; (100,0)*{}};
{\ar@{.>} (115,10)*{}; (115,0)*{}};
{\ar@{.>} (115,10)*{}; (100,10)*{}};
{\ar@{.>} (130,10)*{}; (115,10)*{}};
{\ar@{.>} (115,20)*{}; (115,10)*{}};
(125,15)*{y};
(87,-5)*{x'};
(107,5)*{x''};
\endxy
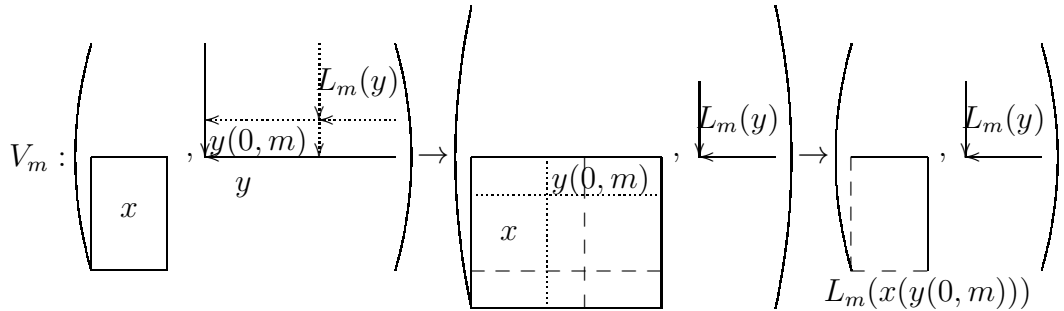
$$

$$
\xy
(-7,0)*{V_{m}:};
{\ar@/^/@{-} (0,-15)*{}; (0,15)*{}};
{(0,0)*{}; (10,0)*{} **\dir{-}};
{(10,0)*{}; (10,-15)*{} **\dir{-}};
{(10,-15)*{}; (0,-15)*{} **\dir{-}};
{(0,0)*{}; (0,-15)*{} **\dir{-}};
(5,-7)*{x};
(22,2)*{y(0,m)};
(13,0)*{,};
{\ar@/_/@{-} (40,-15)*{}; (40,15)*{}};
{\ar (40,0)*{}; (15,0)*{}};
{\ar (15,15)*{}; (15,0)*{}};
{\ar@{.>} (30,5)*{}; (30,0)*{}};
{\ar@{.>} (30,5)*{}; (15,5)*{}};
{\ar@{.>} (30,15)*{}; (30,5)*{}};
{\ar@{.>} (40,5)*{}; (30,5)*{}};
(35,10)*{L_{m}(y)};
(20,-4)*{y};
(45,0)*{\to};
{\ar@/^/@{-} (50,-20)*{}; (50,20)*{}};
{(50,0)*{}; (75,0)*{} **\dir{-}};
{(75,0)*{}; (75,-20)*{} **\dir{-}};
{(75,-20)*{}; (50,-20)*{} **\dir{-}};
{(50,-20)*{}; (50,0)*{} **\dir{-}};
{(60,0)*{}; (60,-20)*{} **\dir{.}};
{(50,-5)*{}; (75,-5)*{} **\dir{.}};
(55,-10)*{x};
(67,-3)*{y(0,m)};
(77,0)*{,};
{\ar@/_/@{-} (90,-20)*{}; (90,20)*{}};
{(65,0)*{}; (65,-20)*{} **\dir{--}};
{(50,-15)*{}; (75,-15)*{} **\dir{--}};
{\ar (90,0)*{}; (80,0)*{}};
{\ar (80,10)*{}; (80,0)*{}};
(85,5)*{L_{m}(y)};
(95,0)*{\to};
{\ar@/^/@{-} (100,-15)*{}; (100,15)*{}};
{(100,0)*{}; (110,0)*{} **\dir{-}};
{(110,0)*{}; (110,-15)*{} **\dir{-}};
{(110,-15)*{}; (100,-15)*{} **\dir{--}};
{(100,-15)*{}; (100,0)*{} **\dir{--}};
(110,-18)*{L_{m}(x(y(0,m)))};
{\ar@/_/@{-} (125,-15)*{}; (125,15)*{}};
{\ar (125,0)*{}; (115,0)*{}};
{\ar (115,10)*{}; (115,0)*{}};
(120,5)*{L_{m}(y)};
(112,0)*{,};
\endxy
$$
\begin{proof}
\begin{enumerate}
\item
We note that $dom(T_{m})=\{(x,y)\in Z:\sigma(x)\geq m\}$ and $dom(V_{m})=Z$ so $dom(T_{m}V_{k})=dom(V_{k}T_{m})=dom(T_{m})$. $T_{m}$ is a local homeomorphism from $Z\cap (R_{\lambda}R_{\lambda}^{*}\times\Lambda^{\infty})$ to $Z\cap(X\times s_{\lambda}s_{\lambda}^{*})$  and $V_{m}$ is a homeomorphism from $\{(x,y)\in Z:xy(0,m)\in L_{\lambda}L_{\lambda}^{*}\}$ to $\{(x,y)\in Z:t(x)=s(\lambda)\}$ with $\sigma(\lambda)=m$ . Using the map $\phi$ from Proposition \ref{mainprop}, we can see the map $(T_{m},V_{k})$ as a homeomorphism from the set $\phi^{-1}(\{(p,x):p\geq m, x\in dom(l_{\lambda})\})$ to $\phi^{-1}(\{(p,x):p\geq 0,x\in dom(l_{s(\lambda)})\})$. The semigroup condition of $T$ and $V$ is easily verified. We prove now that $T_{m}V_{k}=V_{k}T_{m}$. Let $(x,y)\in Z$ with $\sigma(x)\geq m$. Write $x=x'x''$, $y=y'y''$, $x''y'=\alpha\beta$ with $\sigma(x'')=\sigma(\beta)=m$, $\sigma(y')=\sigma(\alpha)=k$.
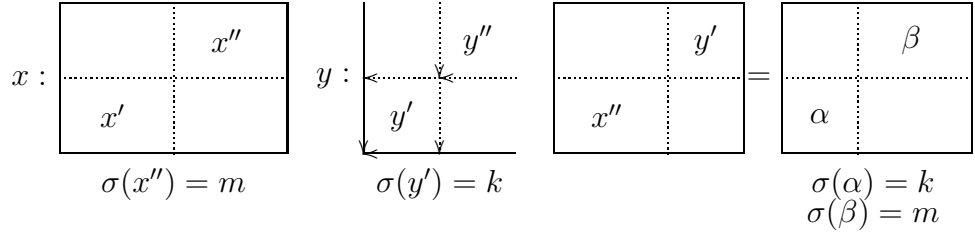
$$
\xy
{(0,0)*{}; (30,0)*{} **\dir{-}};
{(30,0)*{}; (30,20)*{} **\dir{-}};
{(30,20)*{}; (0,20)*{} **\dir{-}};
{(0,20)*{}; (0,0)*{} **\dir{-}};
{(0,10)*{}; (30,10)*{} **\dir{.}};
{(15,0)*{}; (15,20)*{} **\dir{.}};
(15,-4)*{\sigma(x'')=m};
(7,5)*{x'};
(22,15)*{x''};
(-4,10)*{x:};
{\ar (60,0)*{}; (40,0)*{}};
{\ar (40,20)*{}; (40,0)*{}};
{\ar@{.>} (50,10)*{}; (50,0)*{}};
{\ar@{.>} (50,10)*{}; (40,10)*{}};
{\ar@{.>} (50,20)*{}; (50,10)*{}};
{\ar@{.>} (60,10)*{}; (50,10)*{}};
(36,10)*{y:};
(45,5)*{y'};
(55,15)*{y''};
(50,-4)*{\sigma(y')=k};
{(65,0)*{}; (90,0)*{} **\dir{-}};
{(90,0)*{}; (90,20)*{} **\dir{-}};
{(90,20)*{}; (65,20)*{} **\dir{-}};
{(65,20)*{}; (65,0)*{} **\dir{-}};
{(80,0)*{}; (80,20)*{} **\dir{.}};
{(65,10)*{}; (90,10)*{} **\dir{.}};
(72,5)*{x''};
(85,15)*{y'};
(92,10)*{=};
{(95,0)*{}; (120,0)*{} **\dir{-}};
{(120,0)*{}; (120,20)*{} **\dir{-}};
{(120,20)*{}; (95,20)*{} **\dir{-}};
{(95,20)*{}; (95,0)*{} **\dir{-}};
{(105,0)*{}; (105,20)*{} **\dir{.}};
{(95,10)*{}; (120,10)*{} **\dir{.}};
(100,5)*{\alpha};
(112,15)*{\beta};
(107,-4)*{\sigma(\alpha)=k};
(107,-8)*{\sigma(\beta)=m};
\endxy
$$
 One has 
\begin{eqnarray*}
T_{m}V_{k}(x,y)
&=&
T_{m}(L_{k}(x'x''y'),y'')=T_{m}(L_{k}(x'\alpha\beta),y'')\\
&=&
T_{m}(L_{k}(x'\alpha)\beta,y''))=(L_{k}(x'\alpha),\beta y'')\\
V_{k}T_{m}(x,y)
&=&
V_{k}T_{m}(x'x'',y'y'')=V_{k}(x',x''y'y'')\\
&=&
V_{k}(x',\alpha\beta y'')=(L_{k}(x'\alpha),\beta y'')
\end{eqnarray*}
The (DC) condition is satisfied since $dom(T_{m}V_{k})\cap dom(T_{m'}V_{k'})=dom(T_{m})\cap dom(T_{m'})=\{(x,y)\in Z: \sigma(x)\geq m\vee m'\}=dom(T_{m\vee m'})$.
\item
Recall that the map $\sigma$ in section \ref{cuntzlike} was $\sigma(x)=sup\{m:x\in dom(T_{m})\}$. In our case this supremum is $sup\{m:m\leq \sigma(x)\}=\sigma(x)$.
\end{enumerate}
\end{proof}

There is a natural equivariant map between MGDS $(Z, (T,V))$ and the MGDS $(\overline{\mathbb{N}}^{n}\times \Lambda^{\infty},S\times W)$ which gives the algebra $\mathcal{T}^{\otimes r}\otimes C^{*}(\Lambda)$. Here, an equivariant map between two MGDS $(X,T)$ and $(Y,S)$ means a map $\phi:X\to Y$ such that $\phi\circ T_{m}=S_{m}\circ \phi$.

\begin{prop}
Let $(\overline{\mathbb{N}}^{r},S),(\Lambda^{\infty},W)$ be the MGDS of Example \ref{example} (i), (ii). The map $\phi$ of Proposition \ref{mainprop}(v) is an equivariant map between $(Z,(T,V))$ and $(\overline{\mathbb{N}}^{n}\times \Lambda^{\infty},S\times W)$
\end{prop}
\begin{proof}
Regarding domains, it is enough to check that $dom(S_{k}\phi)=dom(\phi T_{k})$ since $W_{k}$ and $V_{k}$ are  defined everywhere.
$$\phi(dom(T_{k})=\phi(\{(x,y)\in Z; \sigma(x)\geq k\})=\{(n,x);n\geq k\}=dom(S_{k})\phi$$

Now it remains to to check that $S_{k}\phi(x,y)=\phi T_{k}(x,y)$ and $W_{k}\phi(x,y)=\phi V_{k}(x,y)$ for $(x,y)\in Z$ with $\sigma(x)\in \mathbb{N}^{n}$ since $\overline{\Lambda}$ is a dense set in $X$ . If $x=x'x''$ with $\sigma(x'')=k$ one has:
\begin{eqnarray*}
S_{k}(\phi(x,y))
&=&S_{k}(\sigma(x),xy)=(\sigma(x)-k,xy)\\
&=&
(\sigma(x'),x'x''y)=\phi(x',x''y)=\phi(T^{k}(x,y))
\end{eqnarray*}
If $y=y'y''$  and $xy'=\alpha\beta$ with $\sigma(y')=\sigma(\alpha)=k$ (so $\sigma(x)=\sigma(\beta)=k$ ) one has
\begin{eqnarray*}
W_{k}(\phi(x,y))&=&
W_{k}(\sigma(x),xy'y'')=W_{k}(\sigma(x),\alpha\beta y'')\\
&=&
(\sigma(x),\beta y'')=(\sigma(\beta),\beta y'')=\phi(\beta,y'')=\phi(V_{k}(x,y))
\end{eqnarray*}
\end{proof}
An equivariant map $\phi:(X,T)\to (Y,S)$ induces a natural map of the associated groupoids: $\hat{\phi}:G(X,T)\to G(Y,S)$, $\hat{\phi}(x,z,y)=(\phi(x),z,\phi(y))$. The next lemma gives a condition for a morphism of r-discrete groupoids to induce a morphism of the corresponding reduced algebras.
\begin{lemma} \label{tomor}
Let $\phi:G_{1}\to G_{2}$ be a morphism of two r-discrete groupoids such that $\phi:G_{1}^{x}\to G_{2}^{\phi(x)}$ is a bijection for any $x\in G_{1}^{0}$ and $\phi_{0}:G_{1}^{0}\to G_{2}^{0}$ is proper. Then $\phi$ is proper and the map  $C_{c}(G_{1})\ni f \mapsto \title{f}=f\circ\phi\in C_{c}(G_{2})$ extends to a morphism $\tilde{\phi}:C^{*}(G_{2})\to C^{*}_{r}(G_{1})$ 
\end{lemma}
\begin{proof}
It is known \cite{re} that in an r-discrete groupoid the range map is a local homeomorphism. Let $r_{1}, r_{2}$ be the range map in $G_{1},G_{2}$, $K\subset G_{2}$ be a compact set such that $r_{2}:K\to r_{2}(K)$ is a homeomorphism. Since $\phi_{0}$ is proper we know that $(\phi_{0})^{-1}(r_{2}(K))$ is compact. Our hypothesis gives that $r_{1}$ is a homeomorphism from $K$ to $r_{1}(K)$.
It is enough to show that  $\tilde{\phi}$ is a morphism between the topological algebras $C_{c}(G_{2})$ and $C_{c}(G_{1})$.
Indeed $\tilde{\phi}$ is continuous since $\phi$ is proper. 
$$\widetilde{f\star g}(t)=(f\star g)(\phi(t))=\sum_{s\in G_{2}^{\phi(s(t))}}f(\phi(t)s)g(s^{-1})$$
We use now the assumption that $\phi:G_{1}^{x}\to G_{1}^{\phi(x)}$ is a bijection and we get
$$\widetilde{f\star g}(t)=\sum_{v\in G_{1}^{s(t)}}f(\phi(t)\phi(v))g(\phi(v)^{-1})=\tilde{f}\star\tilde{g}(t)$$
hence $\widetilde{f\star g}=\tilde{f}\star\tilde{g}$
\end{proof}

A simple examples of such maps is given by transformation groups $\phi:X\rtimes G\to G$, $\phi(x,g)=g$ where $X$ is compact.
\begin{prop}
The previews lemma holds if $G_{1}=G(Z,(T,V))$ and $G_{2}=G(\overline{\mathbb{N}}^{n}\times \Lambda^{\infty},S\times W)$
\end{prop}
\begin{proof}
First let us prove that the lemma holds for the map $\sigma$ and the pairs of groupoids $G_{1}=G(Z,T)$ and $G_{2}=G(\overline{\mathbb{N}}^{n},S)$. To prove the surjectivity let $(\lambda, x)\in Z$, $(\sigma(\lambda),n-m,k)\in G(\overline{\mathbb{N}}^{r},S)$ such that $n\leq \sigma(\lambda)$, $m\leq k$ and $\sigma(\lambda)-n=k-m$. We construct $(\lambda',x')\in Z$ such that $\sigma(\lambda')=k$ and $T_{n}(\lambda,x)=T_{m}(\lambda',x')$. Since $n\leq \sigma(\lambda)$ one can decompose $\lambda=\alpha\beta$ with $\sigma(\beta)=n$. Then $\lambda'=\alpha((\beta x)(0,m)))$ and $x'=(\beta x)(m,\infty)$ have the properties required. To prove the injectivity let $(\lambda,x)\in Z$, $(\lambda',x')\in Z$, $(\lambda'',x'')\in Z$, $n,m,p,q\in \mathbb{N}^{r}$ such that $n-m=p-q$, $\sigma(\lambda')=\sigma(\lambda'')$ $T_{n}(\lambda,x)=T_{m}(\lambda',x')$ and $T_{p}(\lambda,x)=T_{q}(\lambda'',x'')$. Because of the condition (DC) of T we have $T_{n\vee p}(\lambda,x)=T_{m+n\vee p -n}(\lambda',x')$ and $T_{n\vee p}(\lambda,x)=T_{q+n\vee p -p}(\lambda'',x'')$ so $T_{k}(\lambda',x')=T_{k}(\lambda'',x'')$ where $k=m+n\vee p -n=q+n\vee p -p$, hence $\lambda'=\alpha\beta'$, $\lambda''=\alpha\beta''$ with $\sigma(\beta')=\sigma(\beta)=k$ and $\beta' x'=\beta''x''$. The uniqueness of the factorization of a word in $\Lambda^{\infty}$ implies that $\lambda'=\lambda''$ and $x'=x''$.

We prove now that the above lemma is true for $\psi$, $G_{1}=G(Z,V)$ and $G_{2}=G(\Lambda^{\infty},W)$. To prove the surjectivity, let $(\lambda x,n-m,y)\in G(\Lambda^{\infty},W)$ with $\sigma(\alpha)=n$, $\sigma(\beta)=m$, $\lambda x=\alpha z$, $y=\beta z$. One has $\lambda x(0,n)=\alpha \gamma$ so $T_{n}(\lambda,x)=(\gamma,x(n,\infty))$. Let $\lambda'$ be given by the factorization $\beta \gamma=\lambda'\beta'$ with $\sigma(\beta)'=m$ and $x'=\beta'x(n,\infty)$. Then $T_{m}(\lambda',x')=(\gamma,x(n,\infty))=T_{n}(\lambda,x)$ so $((\lambda,x),n-m,(\lambda',x'))\in G(Z,V)$ and $\lambda'x'=\lambda'\beta'x(n,\infty)=\beta\gamma x(n,\infty)=\beta z=y$.
To prove the injectivity let $((\lambda,x),n-m,(\lambda',x')), ((\lambda,x),p-q,(\lambda'',x'')) \in G(Z,V)$ such that $n-m=p-q$, $\lambda'x'=\lambda''x''$. As for $G(Z,T)$ we have $V_{k}(\lambda',x')=V_{k}(\lambda'',x'')$ where $k=m+n\vee p -n=q+n\vee p -p$. From the definition of $V$ one has the factorizations$\lambda'\alpha=\alpha'\gamma$,$\lambda''\beta=\beta'\gamma$, $x'=\alpha y$, $x''=\beta y$ where $\sigma(\alpha)=\sigma(\alpha')=\sigma(\beta)=\sigma(\beta')=k$. Then we have $\alpha'\gamma y=\beta'\gamma y$, hence $\alpha'=\beta'$ and '$x'=x''$. We also have $\lambda'\alpha=\lambda''\beta=\alpha'\gamma$ so $\lambda'=\lambda''$.

We show now that lemma holds for $\phi$, $G_{1}=G(Z,(T,V))$ and $G_{2}=G(\overline{\mathbb{N}}^{n}\times \Lambda^{\infty},S\times W)$.
First we show the injectivity. Let $(\lambda,x),(\lambda',x'),(\lambda'',x'')\in Z$, $n,m,n',m',p'q,p',q'\in \mathbb{N}^{r}$ such that $V_{n}(T_{m}(\lambda,x))=V_{n'}(T_{m'}(\lambda',x'))$, 
$V_{p}(T_{q}(\lambda,x))=V_{p'}(T_{q'}(\lambda'',x''))$, $n-n'=p-p'$, $m-m'=q-q'$, $\sigma(\lambda')=\sigma(\lambda'')$, $\lambda'x'=\lambda''x''$.
As before we have $V_{n}(T_{m\vee q}(\lambda,x))=V_{n'}(T_{m'+m\vee q-m}(\lambda',x'))$ and $V_{p}(T_{m\vee q}(\lambda,x))=V_{p'}(T_{q'+m\vee q-q}(\lambda',x'))$ so 
$V_{n}(T_{m\vee q}(\lambda,x))=V_{n'}(T_{k}(\lambda',x'))$ and $V_{p}(T_{m\vee q}(\lambda,x))=V_{p'}(T_{k}(\lambda'',x''))$ with $k=m'+m\vee q-m=q'+m\vee q-q$. As $\psi(T_{k}(\lambda,x))=\psi(\lambda,x)$ the injectivity result proven above for 
$\psi$, $G(Z,V)$ and $G(\Lambda^{\infty},\times W)$ shows that  
$T_{k}(\lambda',x')=T_{k}(\lambda'',x'')$. Now we apply the injectivity result for  $\sigma$, $G(Z,T)$ and $G(\overline{\mathbb{N}}^{n},S)$ to get $(\lambda',x')=(\lambda'',x'')$

To prove the surjectivity let $(\lambda,x)\in Z$ and $((\sigma(\lambda),\lambda x),(n-n',m-m'),(k,y))\in G(\overline{\mathbb{N}}^{n}\times \Lambda^{\infty},S\times W)$. We apply the surjectivity proven above for $\psi$, $G(Z,V)$ and $G(\Lambda^{\infty},\times W)$ in the point $T_{m}(\lambda,x)$. There is $(\alpha,\beta)\in Z$ such that $\alpha\beta=y$ and $V_{n}(T_{m}(\lambda,x))=V_{n'}(\alpha,\beta)$. It follows that $\sigma(\alpha)=\sigma(\lambda)-m=k-m'$. We define now $\lambda'=\alpha \beta(0,m')$ and $x'=\beta(m',\infty)$. One has $\sigma(\lambda')=k$, $\lambda'x'=y$ and $T_{m'}(\lambda',x')=(\alpha,\beta)$ so $V_{n}T_{m}(\lambda,x)=V_{n'}T_{m'}(\lambda',x')$.
\end{proof}

From Lemma \ref{tomor} we have an induced homomorphism $\tilde{\phi}:\mathcal{T}^{\otimes n}\otimes C^{*}(\Lambda)\to C^{*}(Z,(T,V))$. Proposition \ref{germs} gives a homomorphism $\tilde{\pi}:C^{*}(Z,(T,V))\to C^{*}(\textrm{Germ}(Z,(T,V)))$. But $\textrm{Germ}(Z,(T,V))$ is an open subgroupoid of $G\times\textrm{Germ}(\Lambda^{\infty},W)$. The MGDS $(\Lambda^{\infty},W)$ is essentially free so by Proposition \ref{thesame} we have a map from $C^{*}(\textrm{Germ}(Z,(T,V)))$ to $\mathcal{E}\otimes C^{*}(\Lambda)$. Composing these homomorphisms we get the map from $\mathcal{T}^{\otimes n}\otimes C^{*}(\Lambda)$ to $\mathcal{E}\otimes C^{*}(\Lambda)$. Schematically we view this in terms of groupoids by the following diagram:
$$
\xymatrix
{
G(Z,(T,V))\ar[d]^{\pi}\stackrel{\hat{\phi}}{\longrightarrow}&G(\overline{\mathbb{N}}^{r}\times \Lambda^{\infty}, S\times W)\\
\textrm{Germ}(Z,(T,V))\stackrel{\subset}{\textrm{open}}&G \times \textrm{Germ}(\Lambda^{\infty}),W
}
$$
$\phi$ reverses and $\pi$ preserves the arrow when passing to the corresponding algebras. The bisection that defines $S_{i}$ is the set $A_i=\{((n-e_{i},x),(-e_{i},0),(n,x)): n\in \overline{\mathbb{N}}^{r}, n\geq e_i,x\in\Lambda^{\infty}\}$ and
$$\hat{\phi}^{-1}(A_i)=\bigcup_{\sigma(\lambda)=e_i}\{((x\lambda,y),(-e_{i},0),(x,\lambda y)):(x\lambda,y)\in Z\}$$
We have $W_{i}=\tilde{\pi}(1_{\hat{\phi}^{-1}(A_i)})\in C_{r}^{*}(G(Z,(T,V)))$ so $\tilde{\phi}(S_i\otimes 1)=W_i$. Similarly the bisection that defines $s_{\lambda}$ is the set $B_{\lambda}=\{((n,\lambda x),(0,-\sigma(\lambda)),(n,x)):n\in \overline{\mathbb{N}}^{r}, x\in \Lambda^{\infty},t(x)=s(\lambda)\}$. 
$$\hat{\phi}^{-1}(B_{\lambda})=\{((x'\lambda,y'),(0,-\sigma(\lambda)),(x,y)):(x\lambda,y)\in Z,$$
$$\;\;\;\;\;\;\;\;\;\;\;\;\;\;\;\;\;\;\;\;\;\lambda x=x'\mu,y'=\mu y, \sigma(\mu)=\sigma(\lambda)\}$$
Passing to germs with $\tilde{\pi}$, we get $\tilde{\pi}(1_{\hat{\phi}^{-1}(B_{\lambda})})=V_{\lambda}$.
$\phi$ restricted to the set $\{\Omega\}\times \Lambda^{\infty}$ gives an identification between $G(Z,(T,V))\vert_{\{\Omega\}\times \Lambda^{\infty}}$ and $G(\overline{\mathbb{N}}^{r}\times \Lambda^{\infty}, S\times W)\vert_{\{0\}\times \Lambda^{\infty}}$ so an isomorphism between full corners.

\section{Conclusions}
As a conclusion we want to draw the attention to the results in \cite{em1,em2}. Our groupoid approach to the duality between higher rank graph algebras can lead to a notion of higher rank hyperbolic groups. The higher rank version of a tree is a Euclidean building so the $\tilde{A}_{n}$-groups in \cite{cmsz} have to fall into this class. One can then think of the one-sided and two-sided Toeplitz algebras as subalgebras of $\mathbb{L}(l^{2}(\Gamma))$, where $\Gamma$ is such a group. For example the left-sided Toeplitz algebra associated to the free group $\mathbb{F}_{n}$ is $C(\overline{\mathbb{F}_{n}})\rtimes \mathbb{F}_{n}$ where $\overline{\mathbb{F}_{n}}$ is a compactification of $\mathbb{F}_{n}$ using the left-sided distance $d_{l}(\alpha,\beta)=l(\alpha^{-1}\beta)$, $l$ the word length. Then the right-sided Toeplitz algebras is given by the same algebra $C(\partial \overline{\mathbb{F}_{n}})\rtimes \mathbb{F}_{n}$ but the compactification is given by the right-sided distance $d_{r}(\alpha,\beta)=l(\alpha\beta^{-1})$. The two-sided Toeplitz algebra is the algebra generated by these two in their canonical representation on $\mathbb{L}(l^{2}(\mathbb{F}_{n}))$.

\end{document}